\documentclass{amsart}
\usepackage{amsmath,amssymb,mathrsfs,bbm,yhmath,longtable,mathtools}
\usepackage{amsthm,rotating,color,epsfig,bbm,mdframed,url,MnSymbol}

\usepackage{graphicx}

\date{}
\def\tb{\qopname\relax o{tb}}

\newtheorem{prop}{Proposition}

\theoremstyle{remark}
\newtheorem{rema}{Remark}

\graphicspath{{./pictures/}}	

\author {Ivan Dynnikov and Maxim Prasolov}
\title[Classification of Legendrian knots of topological type~$7_6$]{Classification of Legendrian knots of topological type~$7_6$ with
maximal Thurston--Bennequin number}
\thanks{This work is supported by the Russian Science Foundation under grant~14-50-00005}
\address{\noindent V.A. Steklov Mathematical Institute of Russian Academy of Science, 8 Gubkina Str., Moscow 119991, Russia}
\email{dynnikov@mech.math.msu.su}
\email{0x00002a@gmail.com}

\begin{document}
\maketitle

\begin{abstract}
We classify Legendrian knots of topological type~$7_6$ having
maximal Thurston--Bennequin number confirming the corresponding conjectures of~\cite{chong2013}.
\end{abstract}

This paper provides yet another illustration of the method of~\cite{distinguishing} for distinguishing
Legendrian knots.
The reader is referred to~\cite{distinguishing} for terminology.

If~$R$ is a(n oriented) rectangular diagram of a knot, then by~\emph{exchange class} of~$R$
denoted~$\mathscr E(R)$ we mean the set of all (oriented) rectangular diagrams obtained from~$R$ by exchange
moves.

We use the notation of~\cite{bypasses} for oriented types of stabilizations and destabilizations: $\overrightarrow{\mathrm I}$,
$\overrightarrow{\mathrm{II}}$, $\overleftarrow{\mathrm I}$, and
$\overleftarrow{\mathrm{II}}$.
The following table shows the correspondence with the notation of~\cite{OST}:

\centerline{\begin{tabular}{|l|c|c|c|c|}
\hline
notation of~\cite{bypasses}&
$\overrightarrow{\mathrm I}$&$\overleftarrow{\mathrm I}$&$\overrightarrow{\mathrm{II}}$&$\overleftarrow{\mathrm{II}}$\\\hline
notation of~\cite{OST}&
\emph{X:NE}, \emph{O:SW}&
\emph{X:SW}, \emph{O:NE}&
\emph{X:SE}, \emph{O:NW}&
\emph{X:NW}, \emph{O:SE}\\\hline
\end{tabular}}

A diagram obtained from~$R$ by a stabilization of type~$T$, where~$T\in\bigl\{\overrightarrow{\mathrm I},\overrightarrow{\mathrm{II}},
\overleftarrow{\mathrm I},\overleftarrow{\mathrm{II}}\bigr\}$, is denoted by~$S_T(R)$.
One can see that~$\mathscr E(R_1)=\mathscr E(R_2)$ implies~$\mathscr E\bigl(S_T(R_1)\bigr)=\mathscr E\bigl(S_T(R_2)\bigr)$
(this applies only to knots; in the case of many-component links, one should pay attention to which connected
components of the diagrams are modified by the stabilizations).
So, if~$E$ is an exchange class of oriented rectangular diagrams of a knot and~$R\in E$, then~$S_T(E)=\mathscr E\bigl(S_T(R)\bigr)$
is a well defined exchange class not depending on the concrete choice of~$R$.

By~$\xi_+$ we denote the standard contact structure in~$\mathbb R^3$,
and by~$\xi_-$ the mirror image of~$\xi_+$:
$$\xi_+=\ker(x\,dy+dz),\quad
\xi_-=\ker(x\,dy-dz).$$

By~$\mathscr L_+(R)$ (respectively, $\mathscr L_-(R)$) we denote the equivalence class
of $\xi_+$-Legendrian (respectively, $\xi_-$-Legendrian) knots defined by~$R$.
As one knows (see~\cite{bypasses,OST}) we have~$\mathscr L_+(R_1)=\mathscr L_+(R_2)$ (respectively, $\mathscr L_-(R_1)=\mathscr L_-(R_2)$)
if and only if~$R_1$ and~$R_2$ are related by a sequence of moves of the following kinds:
\begin{enumerate}
\item
exchange moves;
\item
stabilizations and destabilization of types~$\overrightarrow{\mathrm I}$ and~$\overleftarrow{\mathrm I}$ (respectively,
$\overrightarrow{\mathrm{II}}$ and~$\overleftarrow{\mathrm{II}}$).
\end{enumerate}
This implies, in particular, that if~$E$ is an exchange class and~$R\in E$, then~$\mathscr L_+(E)=\mathscr L_+(R)$ (respectively,
$\mathscr L_-(E)=\mathscr L_-(R)$) is a well defined equivalence class of $\xi_+$-Legendrian
(respectively, $\xi_-$-Legendrian) knots not depending on a concrete choice of~$R$.

The $\xi_+$-Legendrian (respectively, $\xi_-$-Legendrian) classes of our interest will be denoted~$7_6^{k+}$ (respectively,
$7_6^{k-}$, $k=1,2,3$, and numbered in the order that they follow in~\cite{chong2013}, see Figure~\ref{7_6-leg-fig}.
In the setting of~\cite{chong2013}
all knots are Legendrian with respect to the standard contact structure, but
each knot type~$K$ is considered together with its mirror
image~$m(K)$. The settings of the present paper are different in that we take the mirror image
of the contact structure, not of the knot. For the reader to easier see the correspondence with
the knots in the atlas~\cite{chong2013} we define the~$\xi_-$-Legendrian classes~$7_6^{k-}$, $k=1,2,3$,
through their mirror images (which are~$\xi_+$-Legendrian classes).

By~$r_\medvert$ we denote:
\begin{enumerate}
\item
in the context of Legendrian knots in~$\mathbb R^3$,
the reflection in the $xz$-plane: $r_\medvert(x,y,z)=(x,-y,z)$;
\item
in the context of rectangular diagrams, the reflection
in a vertical line: $r_\medvert(\theta,\varphi)=(-\theta,\varphi)$.
\end{enumerate}

Finally, $\mu$ denotes:
\begin{enumerate}
\item
in the context of Legendrian knots, the Legendrian mirroring, $\mu(x,y,z)=(x,-y,-z)$;
\item
in the context of rectangular diagrams, the reflection in the origin,
$\mu(\theta,\varphi)=(-\theta,-\varphi)$.
\end{enumerate}

\begin{figure}[ht]
\begin{tabular}{ccccc}
\includegraphics[scale=.42]{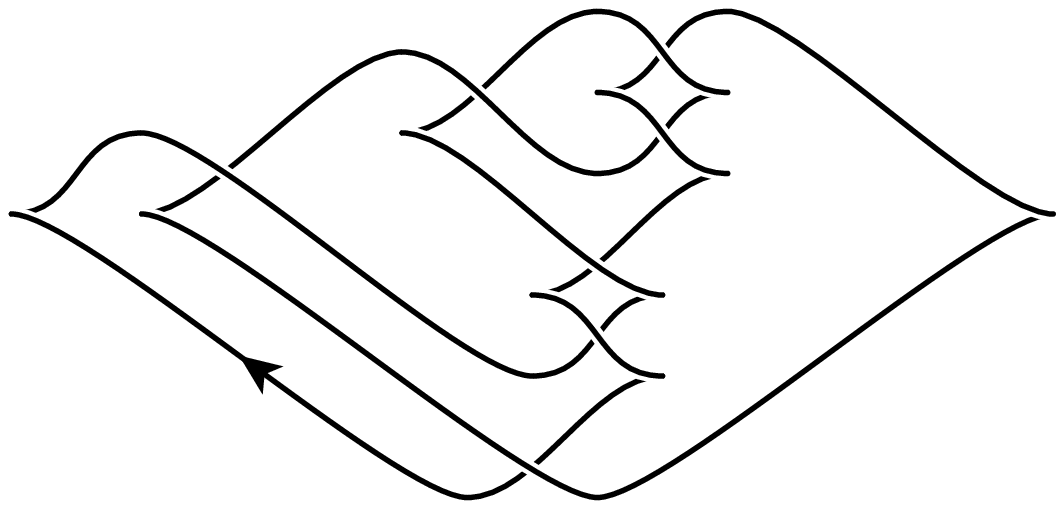}&
\includegraphics[scale=.42]{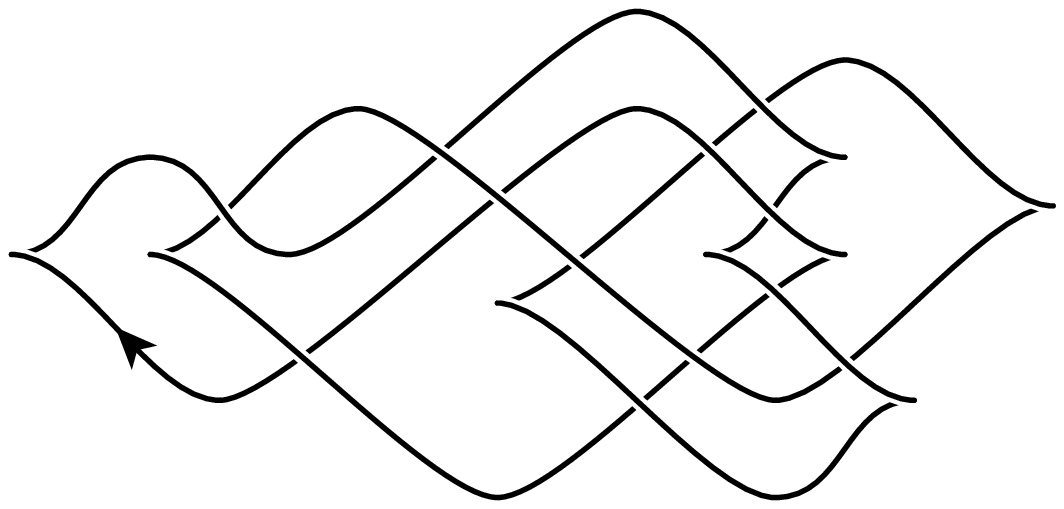}&
\includegraphics[scale=.42]{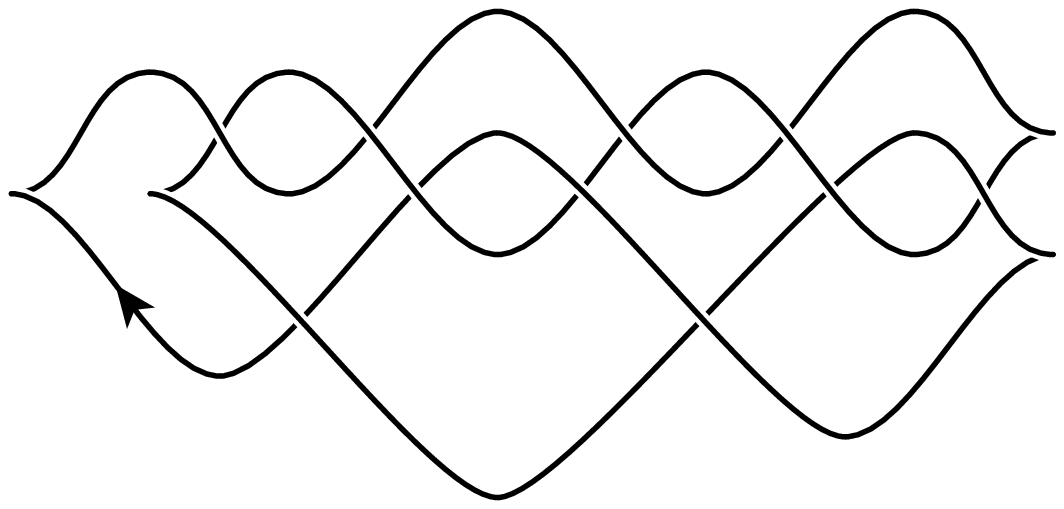}
\\
$7_6^{1+}$&$7_6^{2+}$&$7_6^{3+}$
\\[4mm]
\includegraphics[scale=.42]{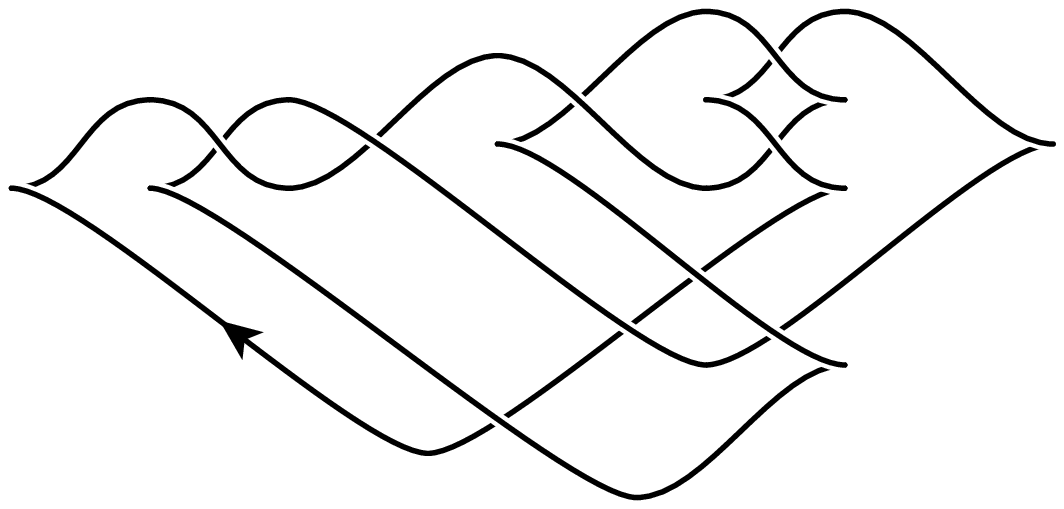}&
\includegraphics[scale=.42]{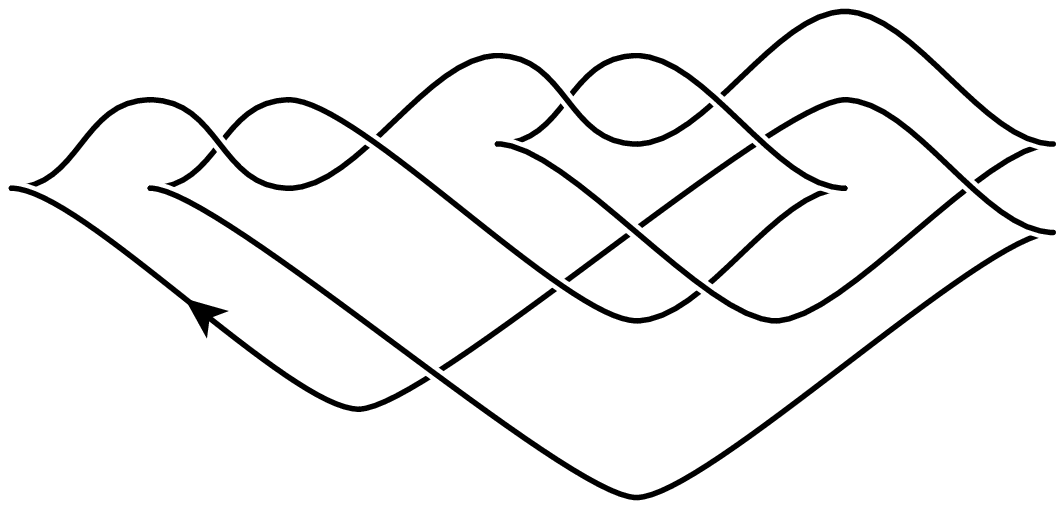}&
\includegraphics[scale=.42]{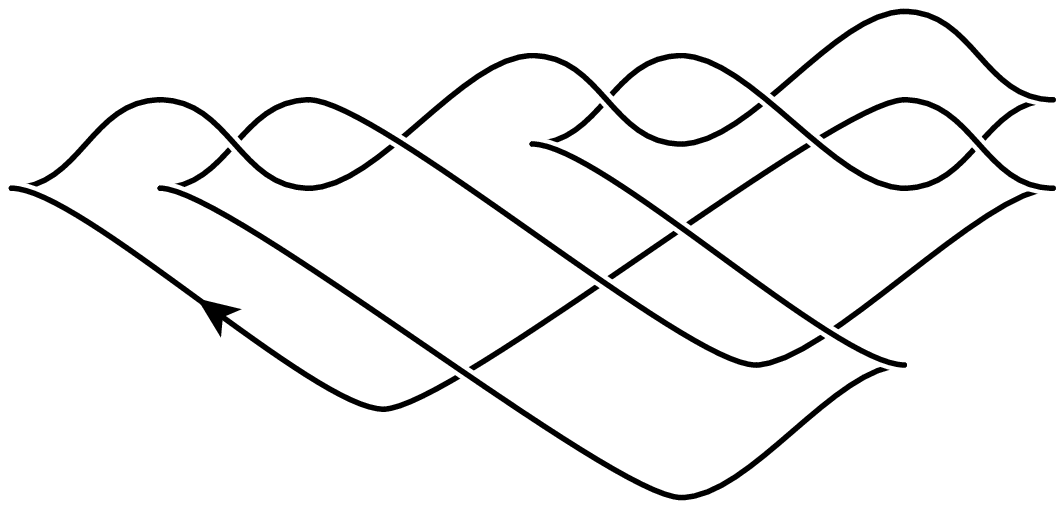}
\\
$r_\medvert(7_6^{1-})$&$r_\medvert(7_6^{2-})$&$r_\medvert(7_6^{3-})$
\end{tabular}
\caption{$\xi_\pm$-Legendrian classes of knots having topological type~$7_6$}\label{7_6-leg-fig}
\end{figure}

\begin{prop}
The following is a complete list, without repetitions, of $\xi_+$-Legendrian classes of
topological type~$7_6$ that have maximal possible Thurston--Bennequin number (which is~$-8$):
\begin{equation}\label{list1-eq}
7_6^{1+}=-\mu(7_6^{1+}),\ -7_6^{1+}=\mu(7_6^{1+}),\ 7_6^{2+}=-\mu(7_6^{2+}),\ -7_6^{2+}=\mu(7_6^{2+}),\ 7_6^{3+},\
-7_6^{3+},\ \mu(7_6^{3+}),\ -\mu(7_6^{3+}).
\end{equation}
\end{prop}

\begin{proof}
It established in~\cite{chong2013} that:
\def\theenumi{\alph{enumi}}
\begin{enumerate}
\item
the list~\eqref{list1-eq} is complete;
\item
$7_6^{1+}=-\mu(7_6^{1+})$, $7_6^{2+}=-\mu(7_6^{2+})$;
\item
each of of~$7_6^{1+}$, $7_6^{2+}$, and~$7_6^{3+}$ has rotation number~$1$,
hence
$$\bigl\{7_6^{1+},7_6^{2+},7_6^{3+},-\mu(7_6^{3+})\bigr\}\cap\bigl\{-7_6^{1+},-7_6^{2+},-7_6^{3+},\mu(7_6^{3+})\bigr\}=\varnothing;$$
\end{enumerate}
It is conjectured but remained unsettled in~\cite{chong2013} that the classes~$7_6^{1+}$,
$7_6^{2+}$, $7_6^{3+}$, and~$-\mu(7_6^{3+})$ are pairwise distinct. To prove this, it suffices
to establish the following two facts:
\begin{enumerate}
\setcounter{enumi}{3}
\item
$7_6^{1+}\ne7_6^{2+}$ and
\item
$7_6^{3+}\ne-\mu(7_6^{3+})$.
\end{enumerate}
In the proof, we use the diagrams~$R_1$--$R_8$ shown in Figure~\ref{7_6-diagrams-fig}.
\begin{figure}[ht]
\begin{tabular}{ccccccc}
\includegraphics[scale=.2]{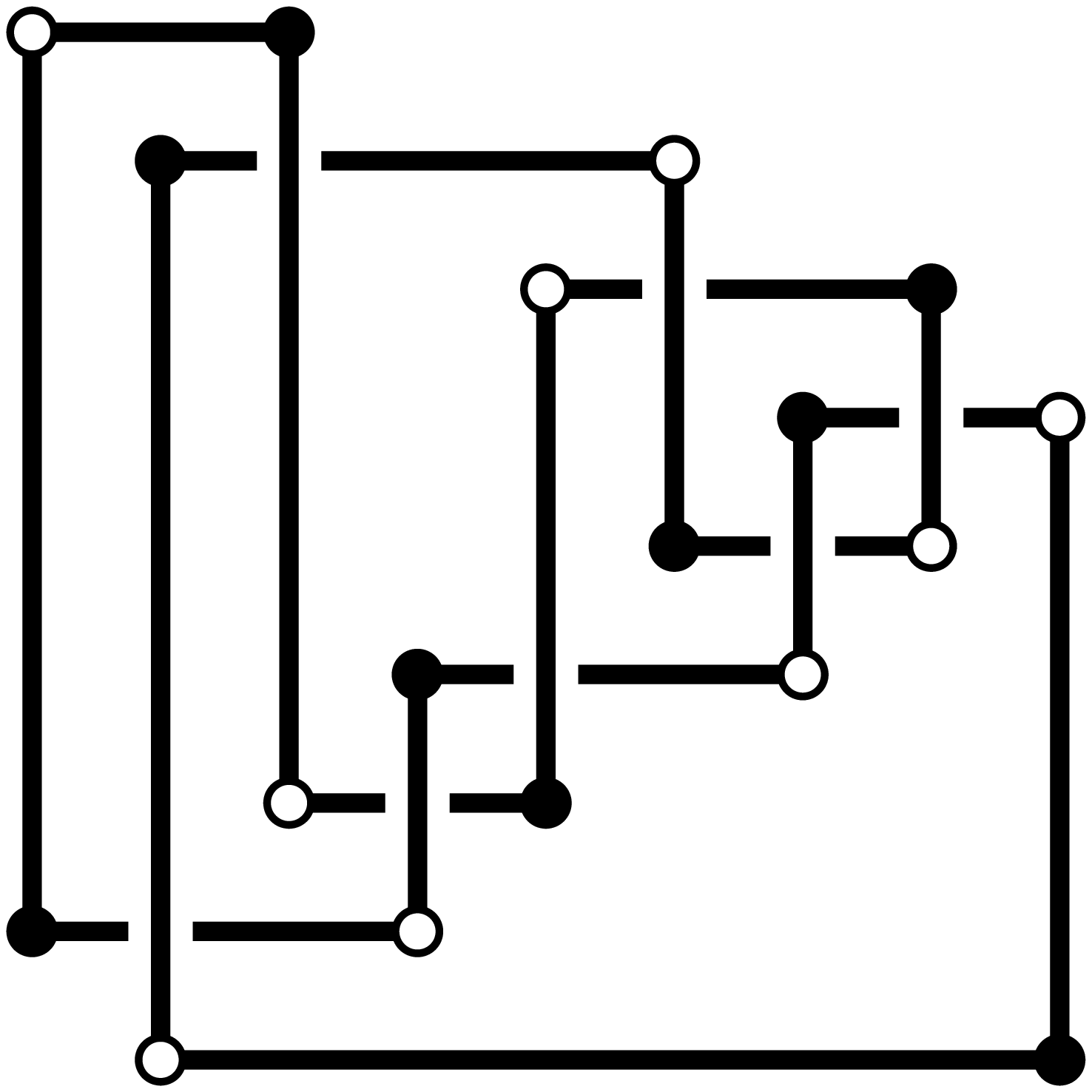}&&
\includegraphics[scale=.2]{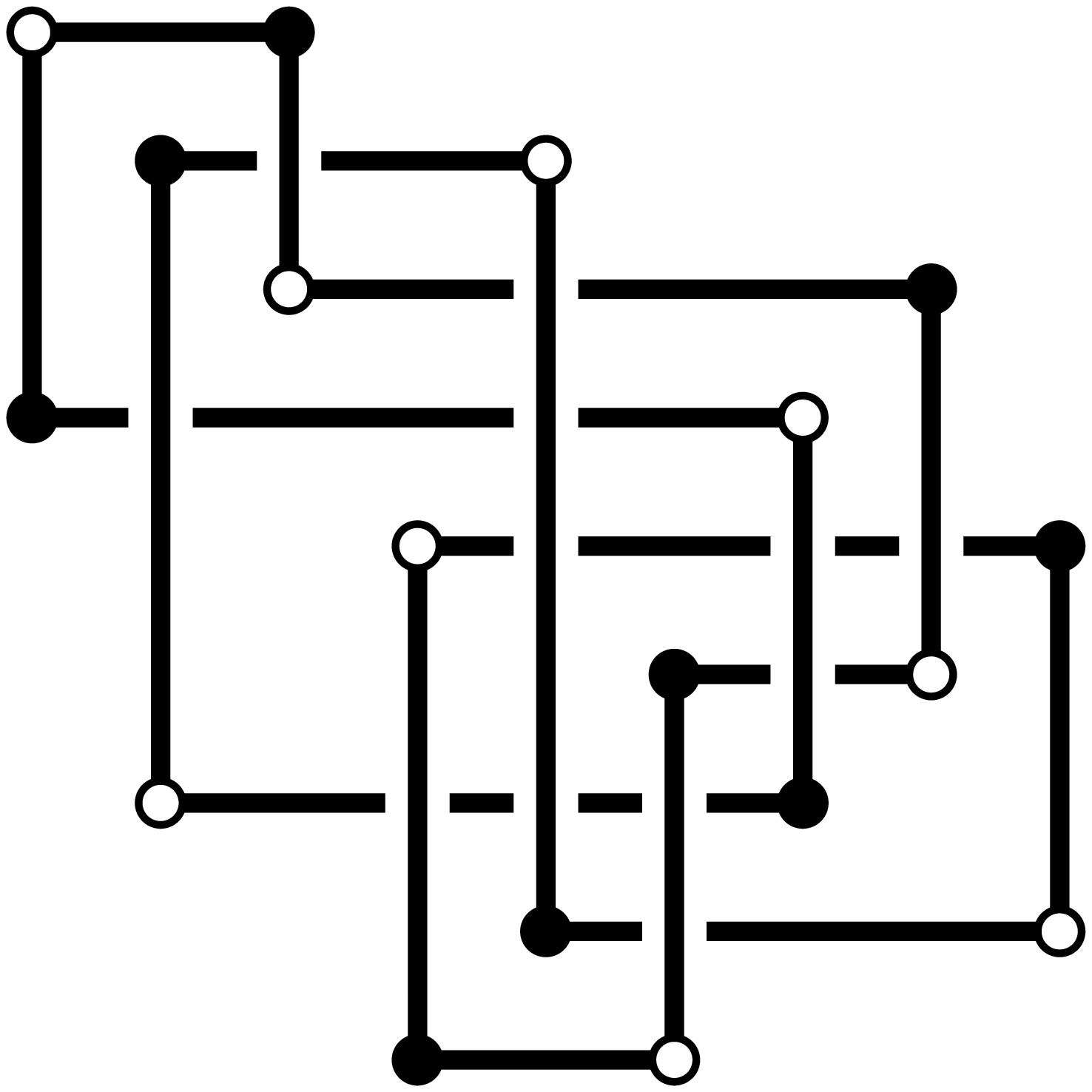}&&
\includegraphics[scale=.2]{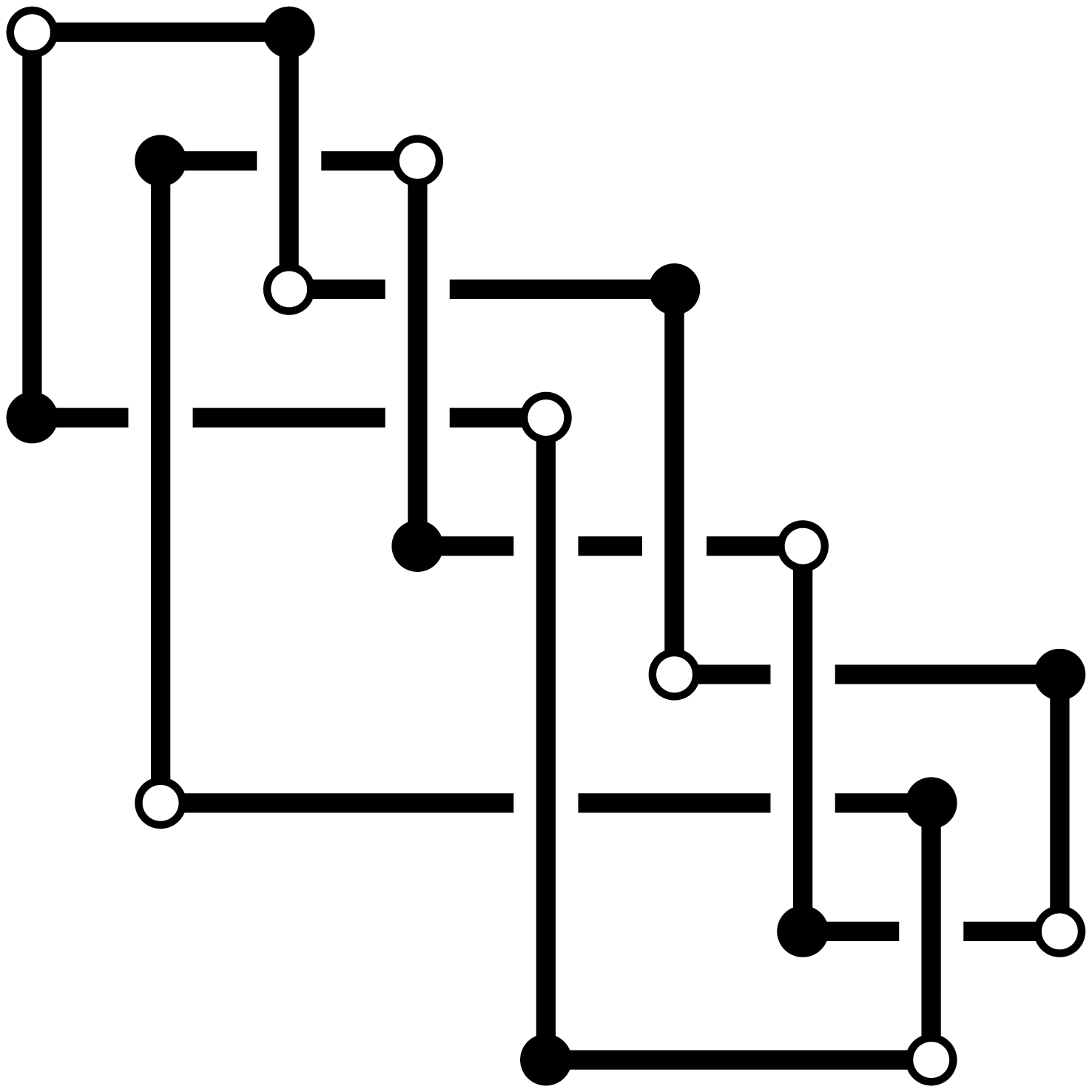}&&
\includegraphics[scale=.2]{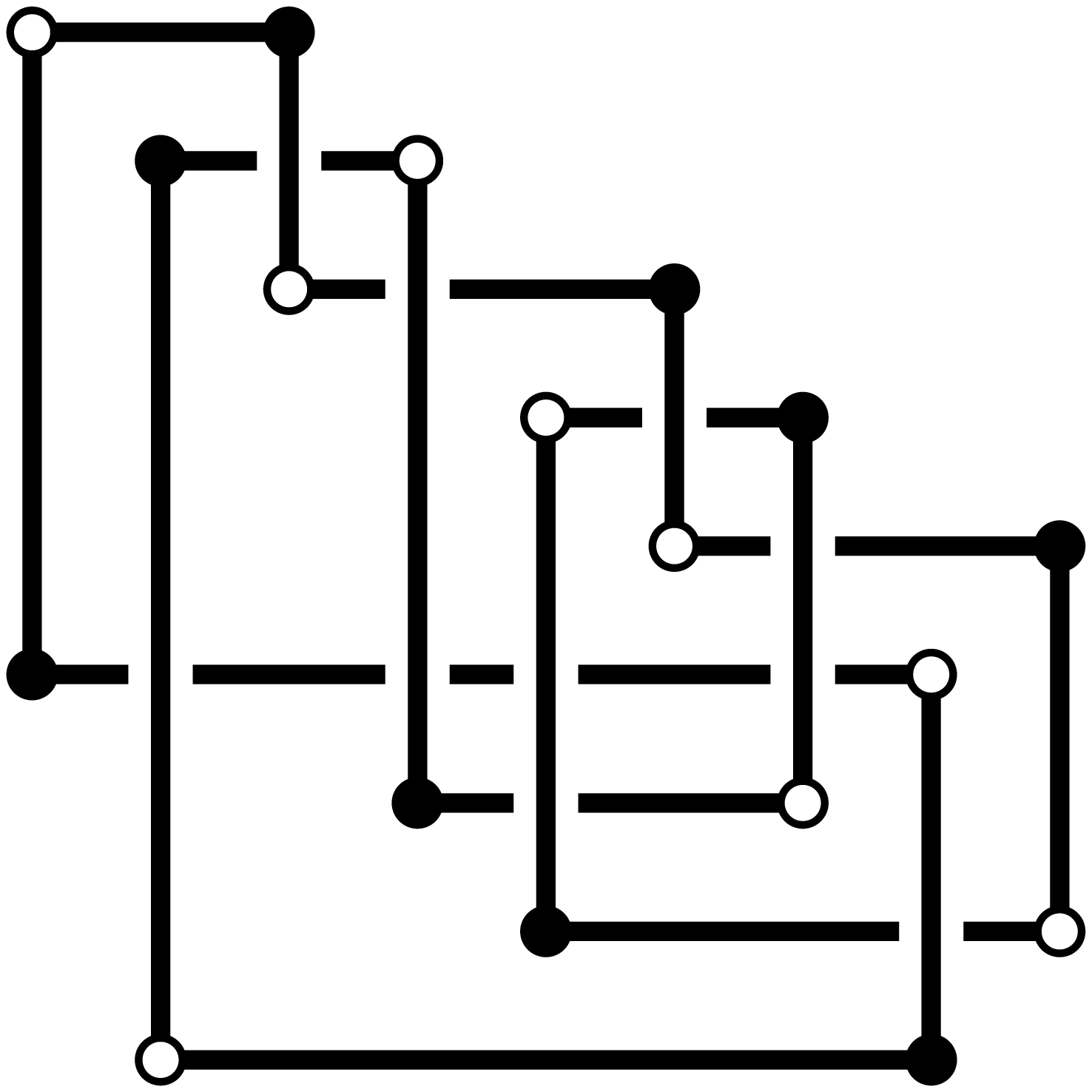}
\\
$R_1$&&$R_2$&&$R_3$&&$R_4$
\\[4mm]
\includegraphics[scale=.2]{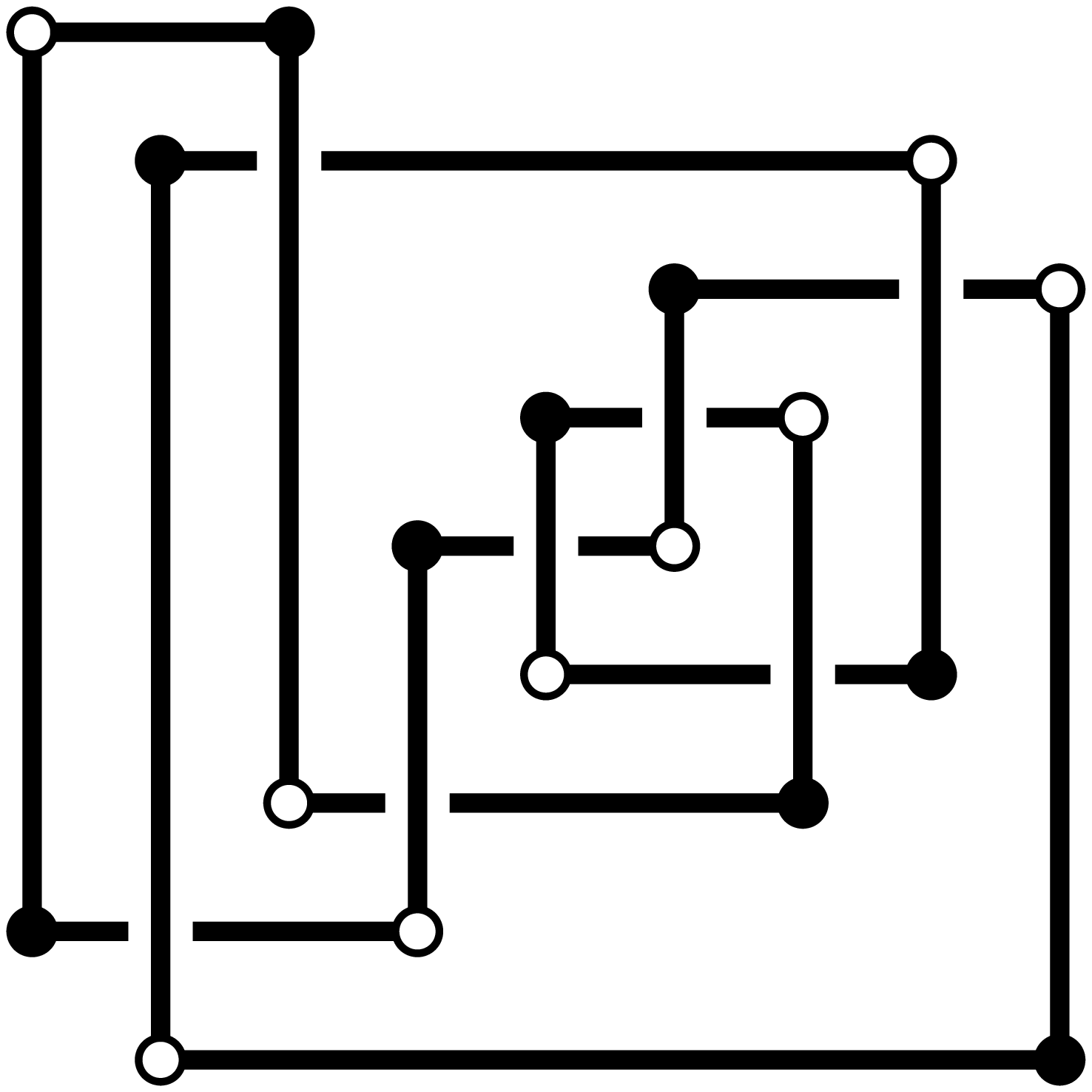}&&
\includegraphics[scale=.2]{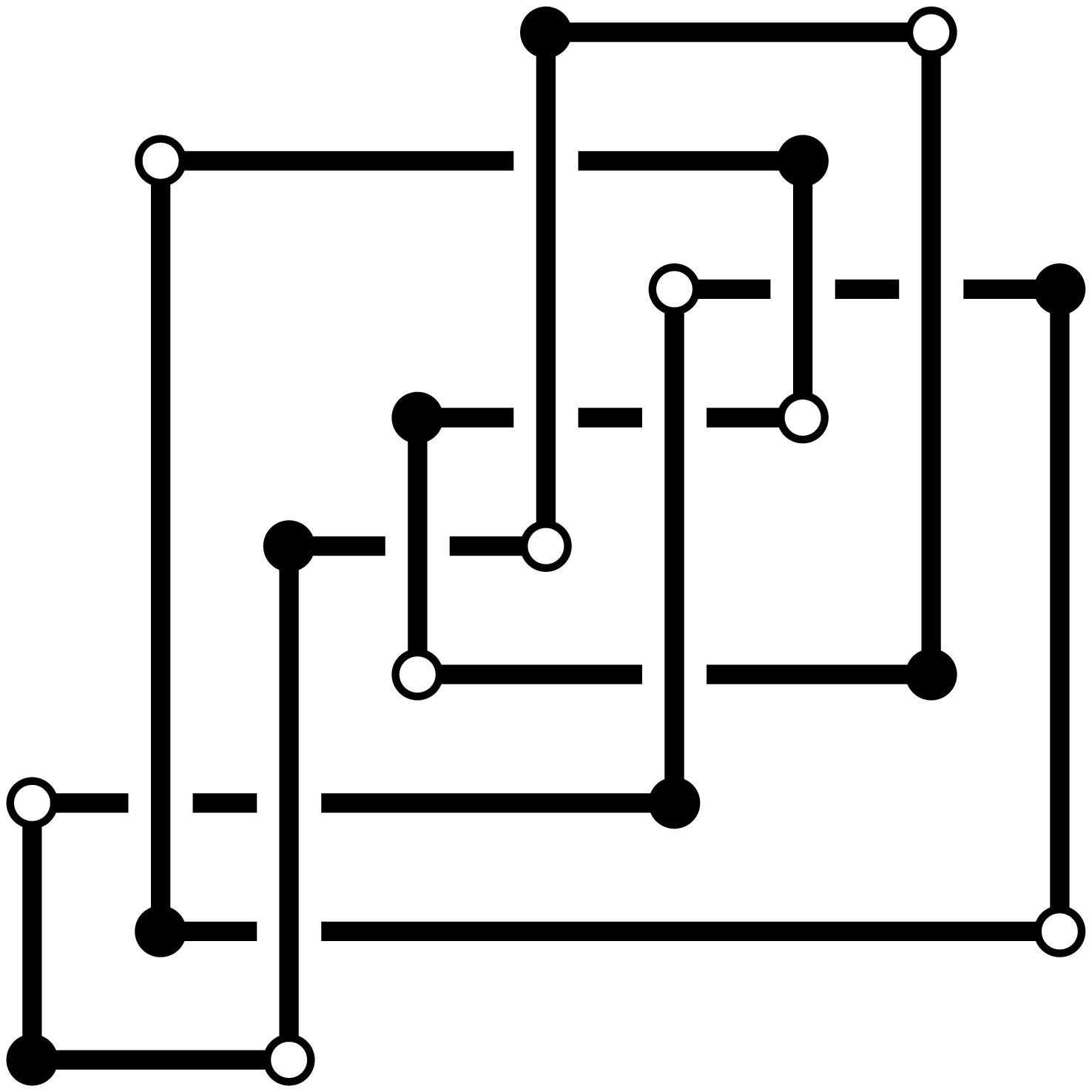}&&
\includegraphics[scale=.2]{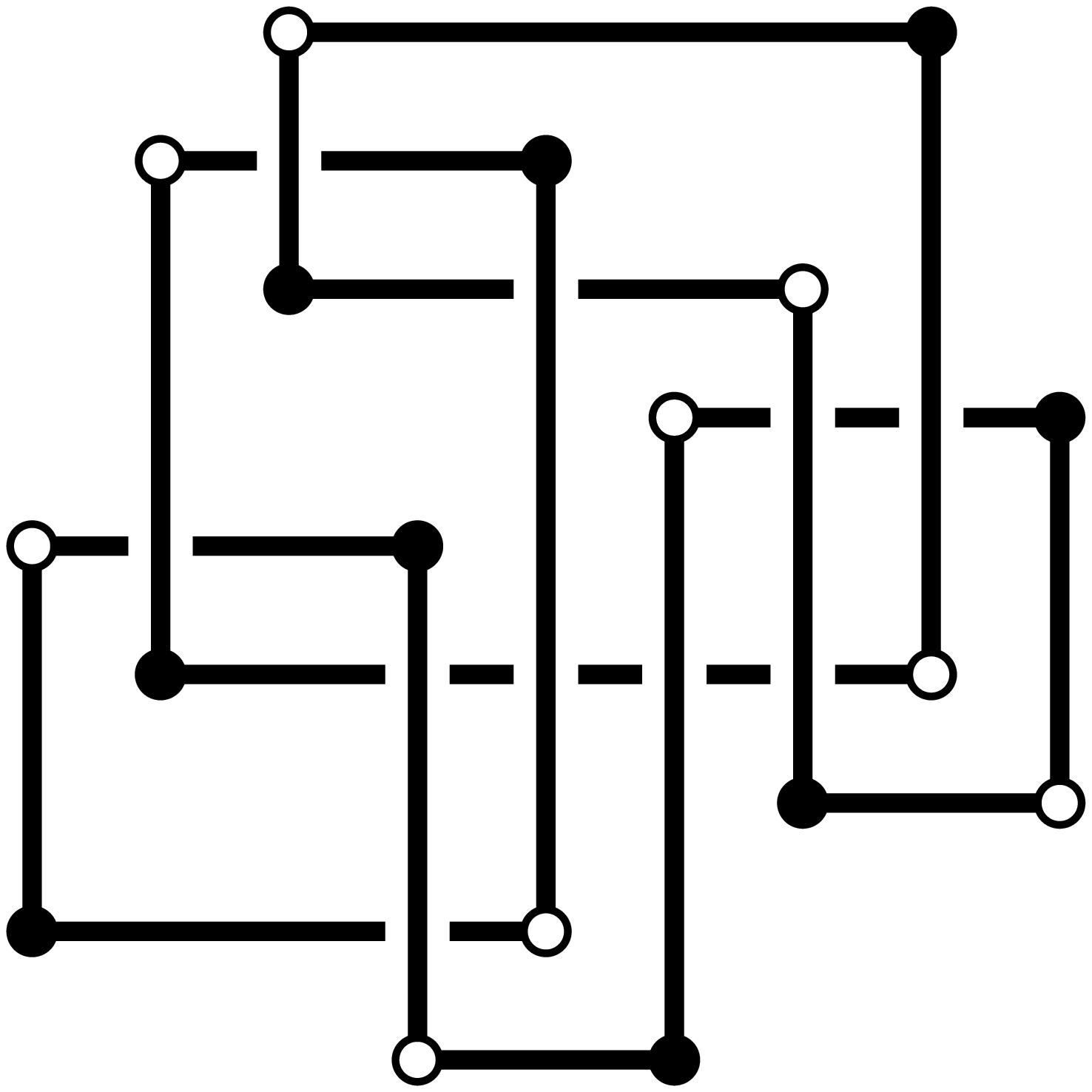}&&
\includegraphics[scale=.2]{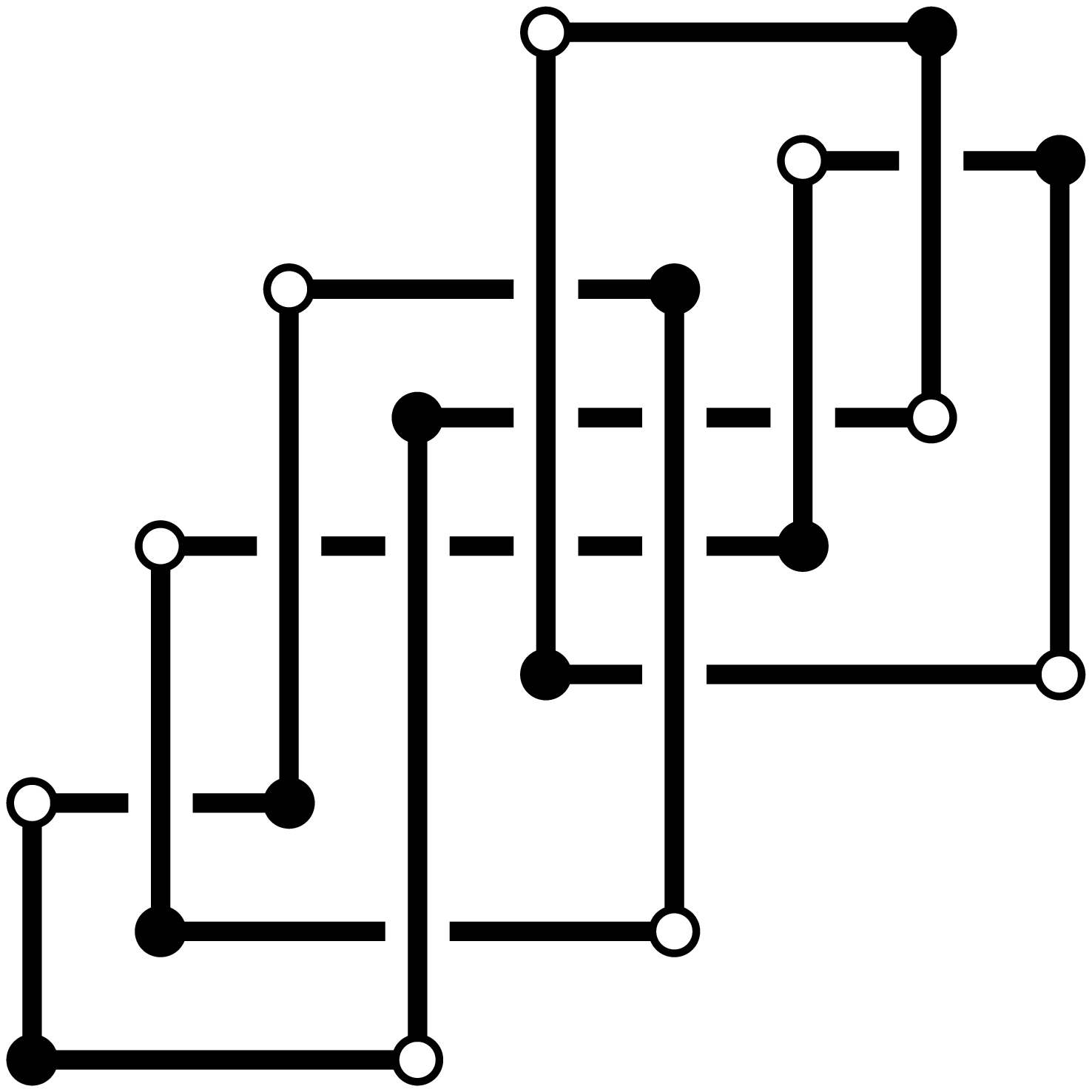}
\\
$R_5$&&$R_6$&&$R_7$&&$R_8$
\end{tabular}
\caption{Rectangular diagrams representing the knot~$7_6$ (or its mirror image). Black
vertices are positive, and white ones are negative}\label{7_6-diagrams-fig}
\end{figure}
To prove each of the statements~(d) and~(e) we follow the lines of the proof of~\cite[Proposition~2.3]{distinguishing}.
Similarly to the~$6_2$ case, the orientation-preserving symmetry group
of the knot~$7_6$ is~$\mathbb Z_2$ (see~\cite{ks92,sak90}),
we denote by~$\sigma$ a self-homeomorphism of~$\mathbb S^3$ representing
the only non-trivial element of this group.
The automorphism of the fundamental group of~$\mathbb S^3\setminus\widehat R_1$
induced by the restriction of~$\sigma$ to~$\mathbb S^3\setminus\widehat R_1$ is denoted by~$\sigma_*$.
(This automorphism is defined up to an internal one. We will make a concrete choice below.)

Like~$6_2$, the knot~$7_6$ is fibered and has genus two~\cite{stal61,mur63}.

\smallskip
\noindent\emph{Proof of~(d)}.
One can immediately see from Figures~\ref{7_6-leg-fig} and~\ref{7_6-diagrams-fig} that
\begin{equation}\label{r1=761+eq}
\mathscr L_+(R_1)=7_6^{1+}\quad and\quad\mathscr L_+(R_2)=7_6^{2+}.
\end{equation}

It is a direct check that~$\mathscr E\bigl(S_{\overrightarrow{\mathrm{II}}}(R_1)\bigr)=\mathscr E\bigl(S_{\overrightarrow{\mathrm{II}}}(R_6)\bigr)$
and~$\mathscr E\bigl(S_{\overleftarrow{\mathrm I}}(R_2)\bigr)=\mathscr E\bigl(S_{\overleftarrow{\mathrm I}}(R_6)\bigr)$, which
implies
\begin{equation}\label{r6=762+eq}
\mathscr L_+(R_6)=7_6^{2+},\quad\mathscr L_-(R_6)=\mathscr L_-(R_1)
\end{equation}
(the class~$\mathscr L_-(R_1)$ coincides with $7_6^{3-}$, which does not play a role here).

One also finds that
$$\tb_+(R_1)=\tb_+(R_2)=\tb_+(R_6)=-8,\quad
\tb_-(R_1)=\tb_-(R_2)=\tb_-(R_6)=-1,$$
hence, any Seifert surface for the knot~$7_6$ is $+$-compatible and $-$-compatible with any of~$R_1$, $R_2$, and~$R_6$
(see \cite[Definition~2.6]{distinguishing}).

Now we choose a Seifert surface for~$\widehat R_1$. Our choice is shown, in the rectangular form,
in Figure~\ref{seifert1-fig} together with the torus projections of the chosen generators of the fundamental
group of the surface.
\begin{figure}[ht]
\includegraphics[scale=.65]{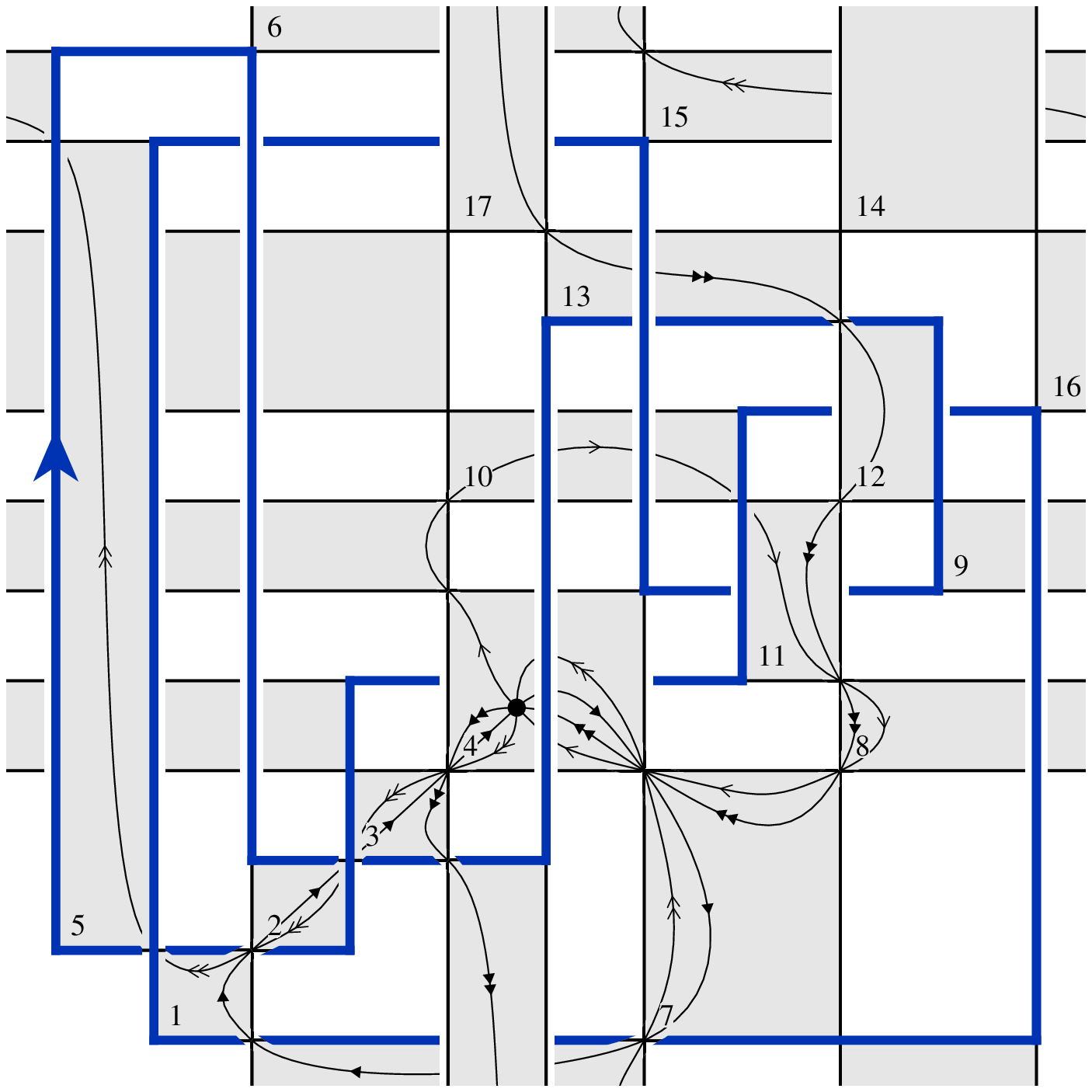}\hskip1cm\raisebox{80pt}{\includegraphics[scale=0.7]{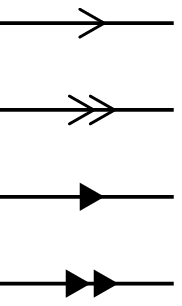}\put(-32,75){legend:}%
\put(3,60){$x_1$}\put(3,42.3){$x_2$}\put(3,24.6){$x_3$}\put(3,6.9){$x_4$}}
\caption{Rectangular diagram~$\Pi_1$ with~$\partial\Pi_1=R_1$}\label{seifert1-fig}
\end{figure}
It is a direct check that~$\widehat\Pi_1$ is orientable and has genus two.
One can also see that the homotopy class of~$\partial\widehat\Pi_1=\widehat R_1$
in~$\widehat\Pi_1$ is presented by the element
\begin{equation}\label{xxxxxxxx-eq}
x_1x_2x_3x_4x_1^{-1}x_2^{-1}x_3^{-1}x_4^{-1}.
\end{equation}

The generators~$x_i$ are chosen so as to have
\begin{equation}\label{sigma1-eq}
\sigma_*(x_i)=x_i^{-1},\quad i=1,2,3,4,
\end{equation}
which will be seen in a moment. They are also shown in Figure~\ref{loops1-fig} with
\begin{figure}[ht]
\includegraphics[scale=.8]{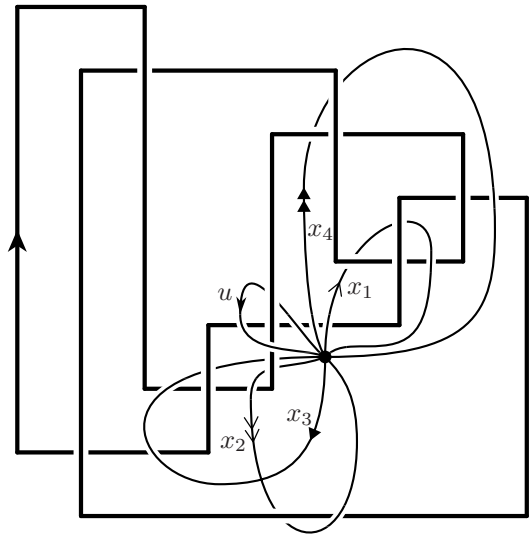}\put(-129,94){$u$}\put(-80,96){$x_1$}\put(-95,118){$x_4$}\put(-128,38){$x_2$}\put(-103,48){$x_3$}
\caption{The generators~$x_1$, $x_2$, $x_3$, $x_4$, and~$u$ of~$\pi_1(\mathbb S^3\setminus\widehat R_1)$}\label{loops1-fig}
\end{figure}
an additional generator~$u$ such that
\begin{equation}\label{sigma2-eq}
\sigma_*(u)=x_2^{-1}u.
\end{equation}
One can verify, using the Wirtinger presentation of~$\pi_1(\mathbb S^3\setminus
\widehat R_1)$, that~$x_1$, $x_2$, $x_3$, $x_4$, and~$u$ generate the fundamental group of~$\widehat R_1$,
and the following list can be taken for a set of defining relations:
$$ux_1u^{-1}=x_4^2x_2^{-1},\quad ux_2u^{-1}=x_2x_4^{-1}x_2x_3x_2x_4^{-1},\quad ux_3u^{-1}=x_4x_2^{-1},\quad ux_4u^{-1}=x_4x_1^{-1}x_4x_2^{-1}.$$
These relations are clearly preserved by the substitution~$x_i\mapsto x_i^{-1}$ ($i=1,2,3,4$), $u\mapsto x_2^{-1}u$, which, therefore,
defines an automorphism of~$\pi_1(\mathbb S^3\setminus\widehat R_1)$. This automorphism is an involution that
preserves the conjugacy class of the element~\eqref{xxxxxxxx-eq} and the homology class~$[u]\in H_1(\mathbb S^3\setminus\widehat R_1;\mathbb Z)$.
This implies that this automorphism is induced by a self-homeomorphism
of~$\mathbb S^3\setminus\widehat R_1$ taking~$\widehat\Pi_1\setminus\widehat R_1$
to itself and preserving the orientations of~$\mathbb S^3$ and~$\widehat\Pi_1$.
Such a homeomorphism must be isotopic to~$\sigma$ (restricted to~$\mathbb S^3\setminus\widehat R_1$),
which verifies~\eqref{sigma1-eq} and~\eqref{sigma2-eq}.

This means, that~$\sigma$ can be chosen so as to have~$\sigma(\widehat\Pi_1)=\widehat\Pi_1$ and $\sigma^2=\mathrm{id}$.
We may also assume that, for each~$i=1,2,3,4$, the homeomorphism~$\sigma$ takes a loop representing~$x_i$
to the inverse of itself. We cut~$\widehat\Pi_1$ along these loops
to get an octagon with a hole. Shown in Figure~\ref{nuts1-fig} on the left
\begin{figure}[ht]
\includegraphics[scale=.5]{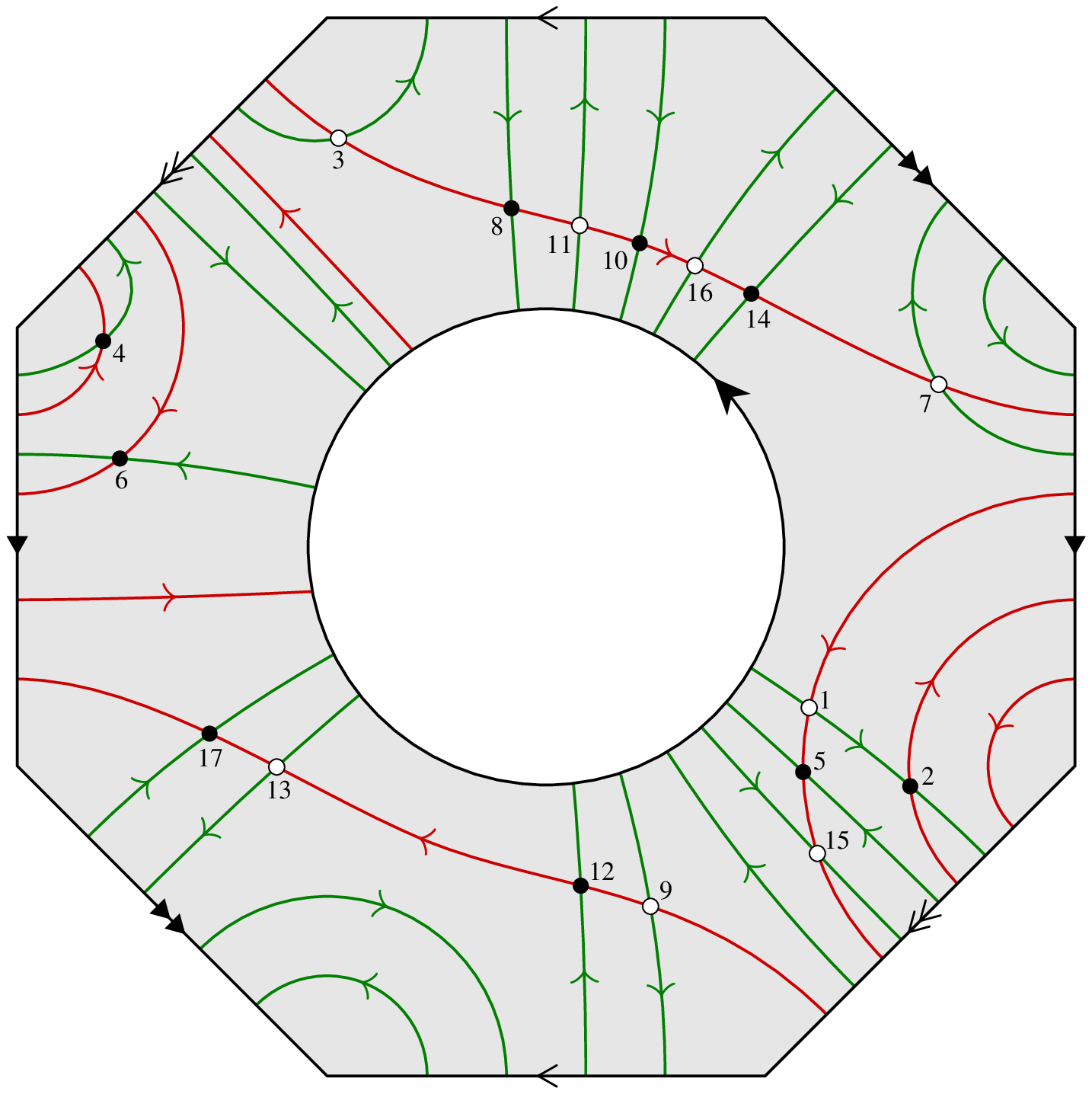}
\includegraphics[scale=.5]{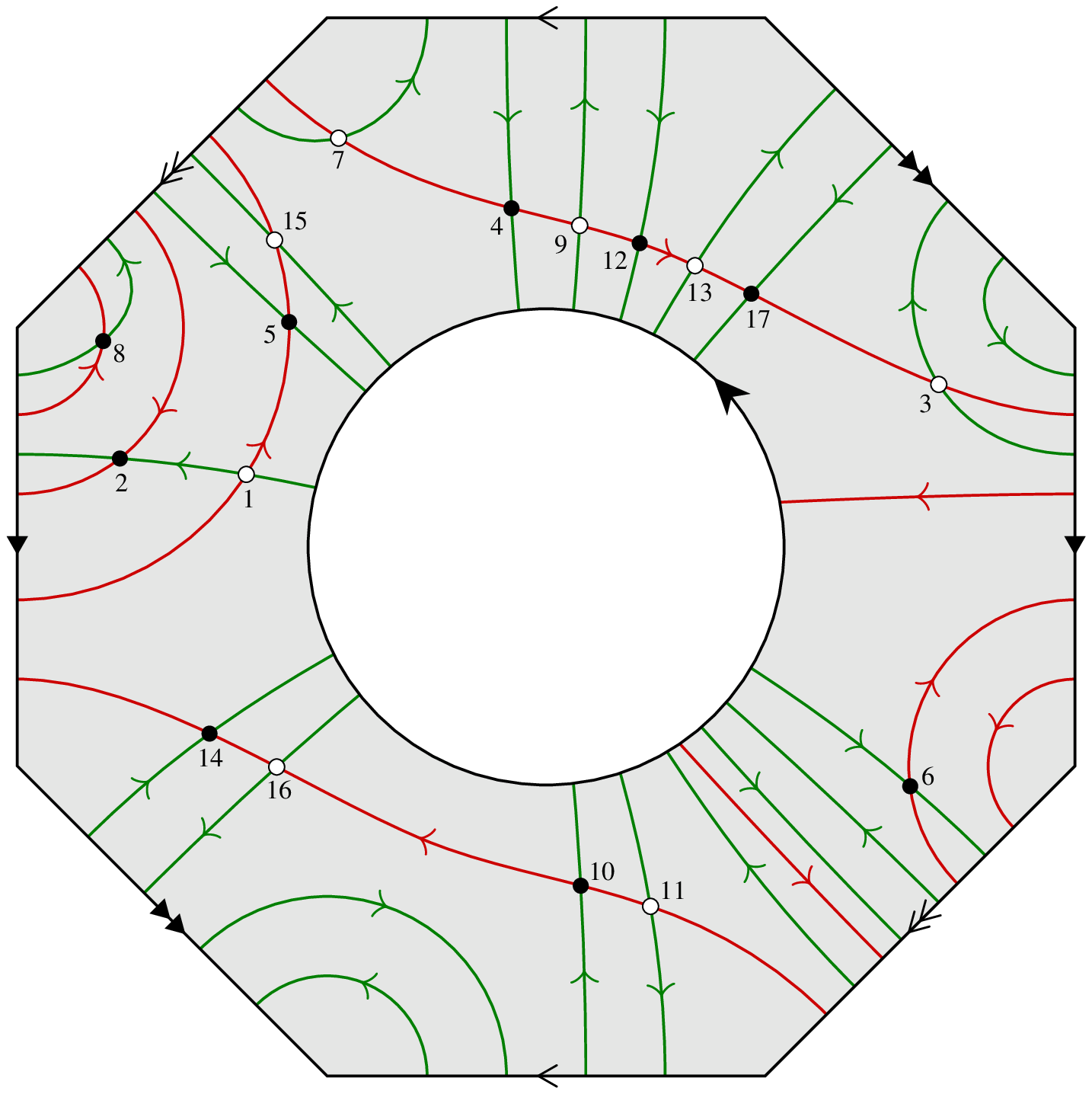}
\caption{Dividing configurations~$(\delta_1^+,\delta_1^-)$ and~$(\delta_1^+,\sigma(\delta_1^-))$}\label{nuts1-fig}
\end{figure}
is a canonic dividing configuration, which we denote by~$(\delta_1^+,\delta_1^-)$, on the cut surface,
with~$\delta_1^+$ shown in green and~$\delta_1^-$ in red.
The right picture in Figure~\ref{nuts1-fig} shows the dividing configuration~$(\delta_1^+,\sigma(\delta_1^-))$
(for a specific choice of~$\sigma$). One can see that both dividing configurations have the same
dividing code, which is
\begin{multline}\nonumber
\{(1,2,3,4),(5),(6,7,8),(9,10),(11,12),(13,14),(15),(16,17)\},\\
\nonumber
\{(1,5,15,6,1),(2),(3,8,11,10,16,14,7,4,9,12,13,17,3)\}.
\end{multline}

The set~$\{\delta_1^-,\sigma(\delta_1^-)\}$ is $-$-representative for~$R_1$,
and hence, for~$R_6$ (see~\cite[Definition~2.8]{distinguishing}).
In view of~\eqref{r1=761+eq} and~\eqref{r6=762+eq}, by~\cite[Theorem~2.1 and Corollary~2.1]{distinguishing}
the equality~$7_6^{1+}=7_6^{2+}$ would imply the existence of a proper realization~$(\Pi,\phi)$ of~$(\delta_1^+,\delta_1^-)$
or~$(\delta_1^+,\sigma(\delta_1^-))$
such that~$\partial\Pi$ is exchange-equivalent to~$R_6$. Since~$(\delta_1^+,\delta_1^-)$ and~$(\delta_1^+,\sigma(\delta_1^-))$
are isomorphic, they have the same sets of realizations (if not requested to be proper).

An exhaustive search (using the script~\cite{dyn-script}) results in exactly four, up to combinatorial equivalence,
realizations~$(\Pi,\phi)$ of~$(\delta_1^+,\delta_1^-)$ such that~$\partial\Pi$ is a rectangular diagram representing a knot
of topological type~$7_6$.
These are shown in Figure~\ref{4realizations-fig}.
\begin{figure}[ht]
\includegraphics[scale=.5]{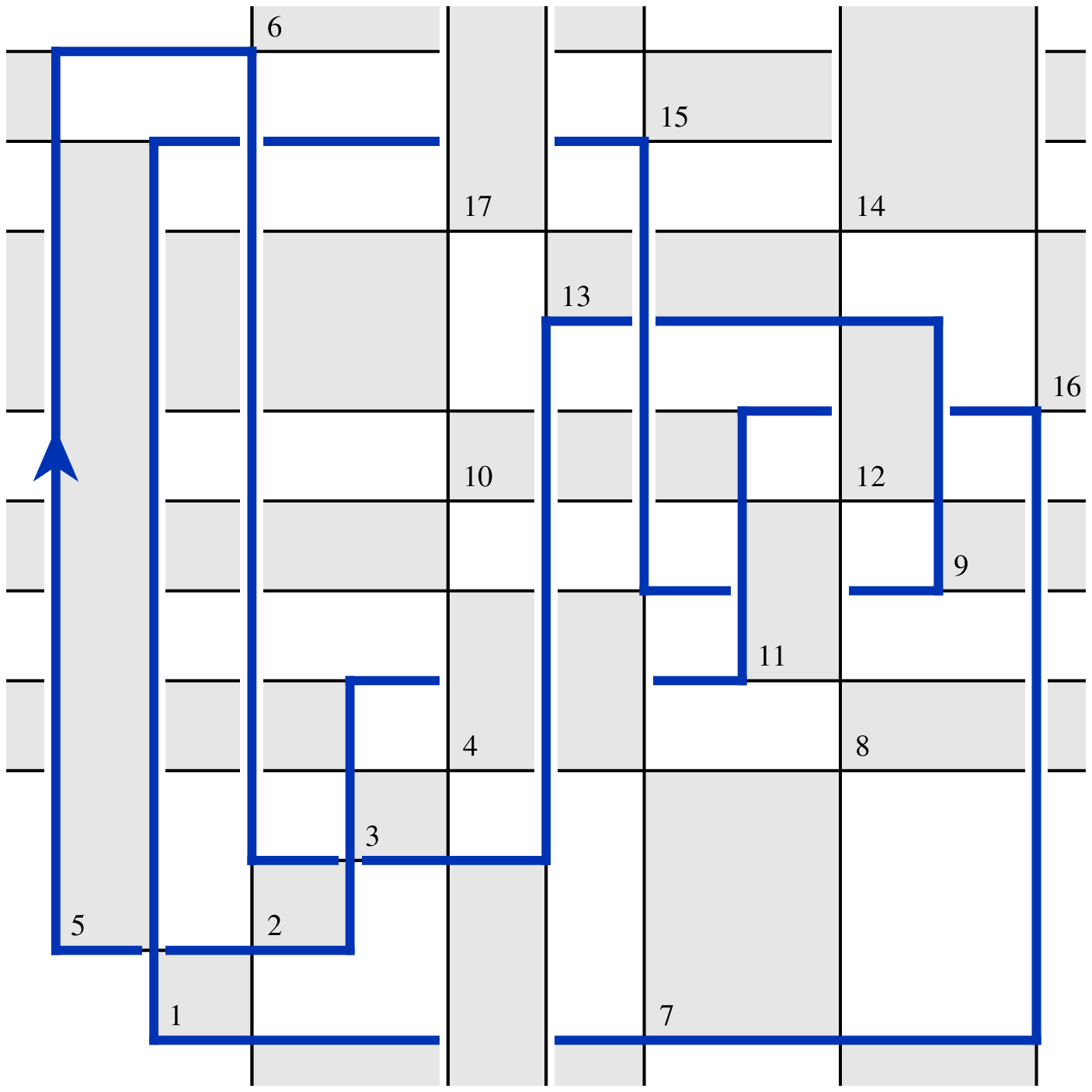}%		58 1 0
\includegraphics[scale=.5]{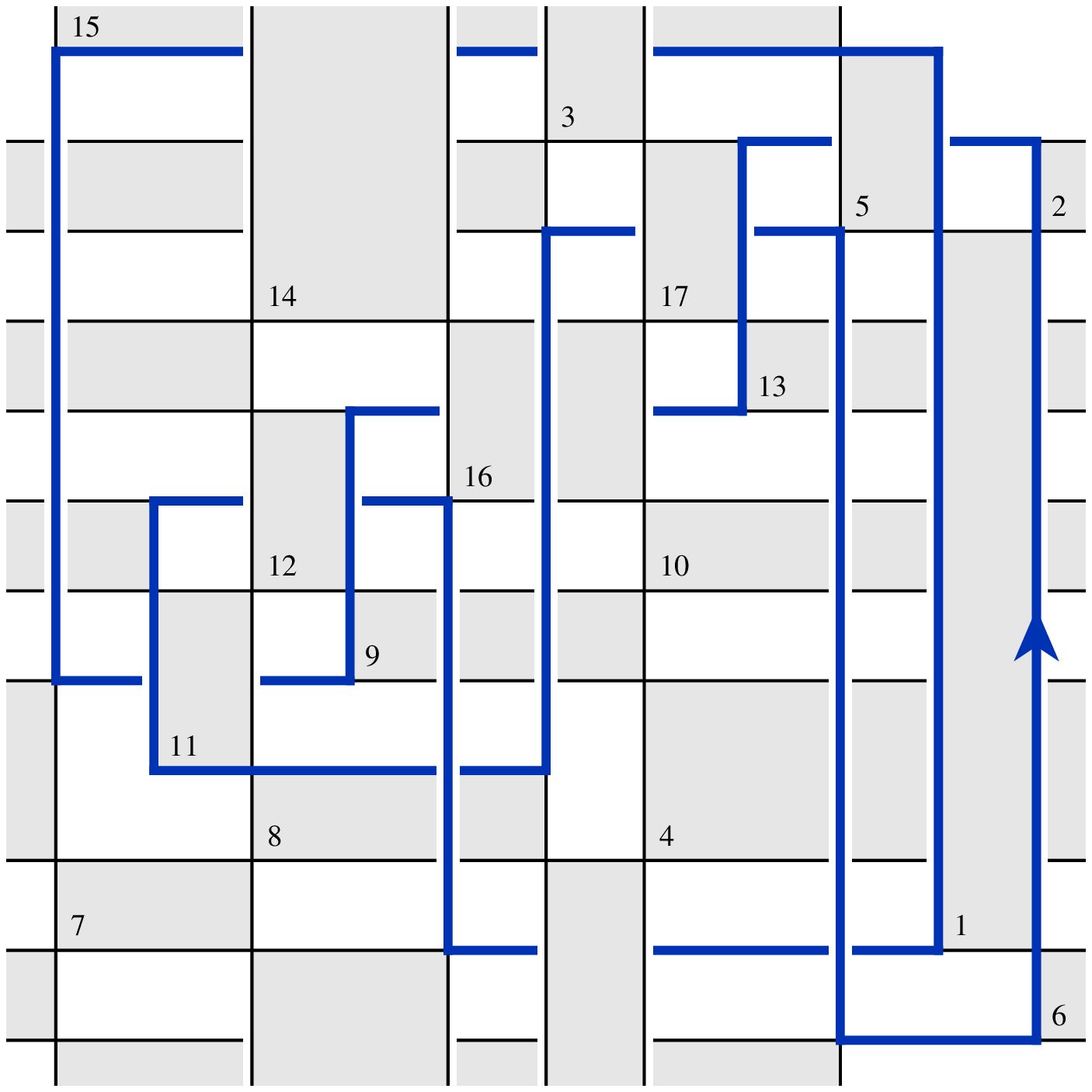}%		51 1 0

\includegraphics[scale=.5]{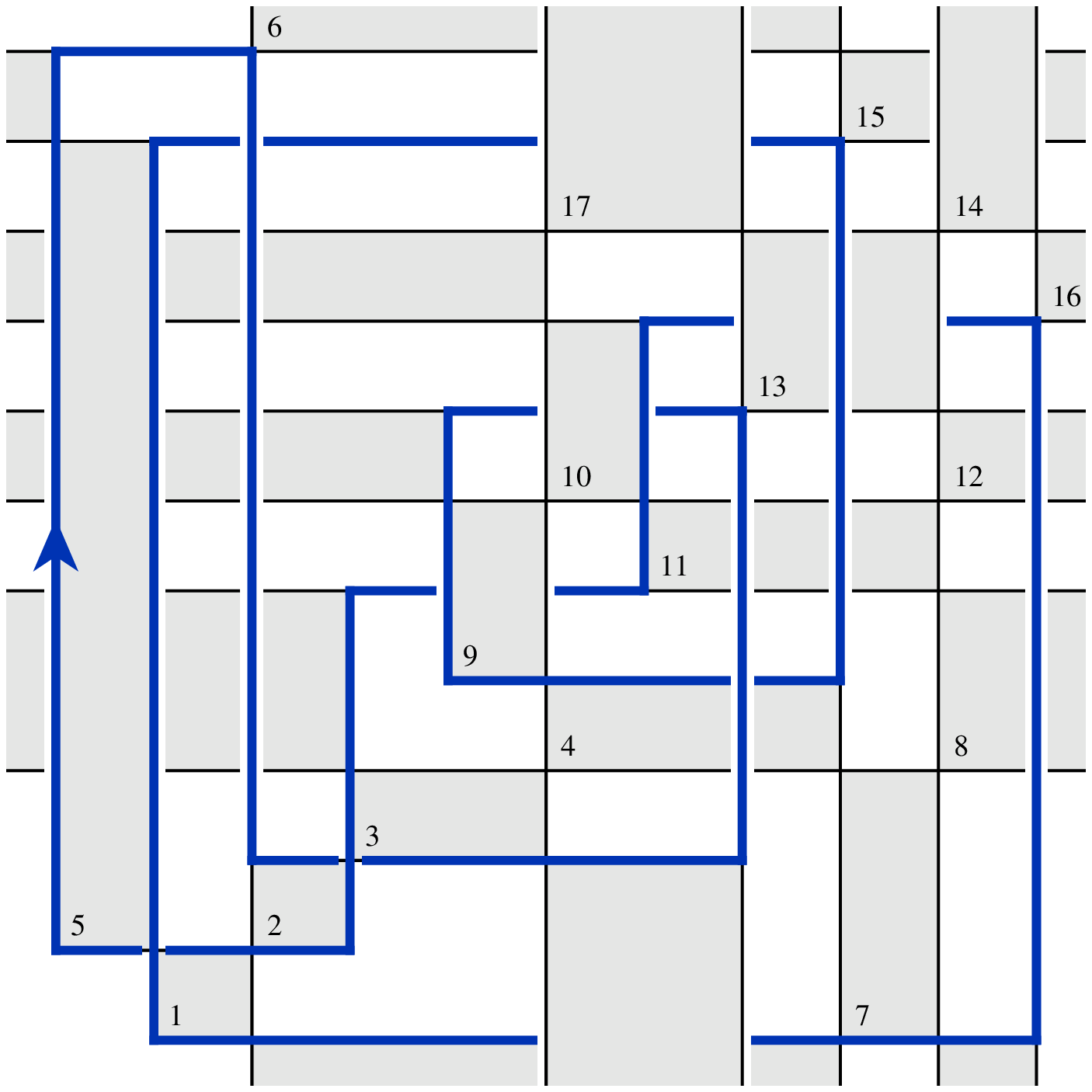}%		59 1 3
\includegraphics[scale=.5]{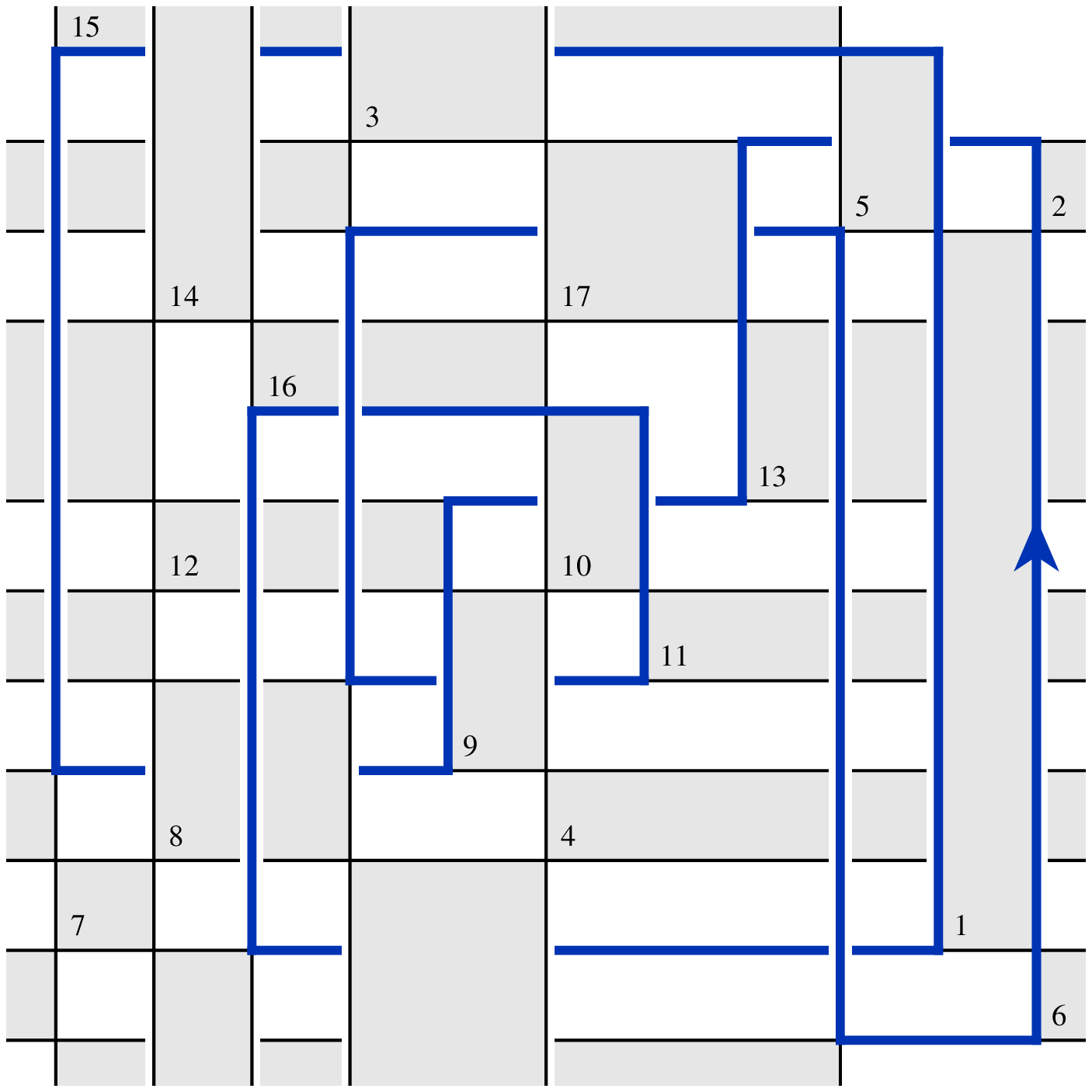}%		43 1 3
\caption{All realizations~$\Pi$ of~$(\delta_1^+,\delta_1^-)$ with~$\partial\Pi$ representing the knot~$7_6$}\label{4realizations-fig}
\end{figure}
The boundaries of the obtained rectangular diagrams of a surface are~$R_1$, $-\mu(R_1)$, $R_5$,
and~$-\mu(R_5)$. None of these rectangular diagrams of a knot admits
a non-trivial exchange move, and none of them is combinatorially equivalent to~$R_6$.
Thus, $R_6\notin\mathscr E(R_1)\cup\mathscr E(R_5)\cup\mathscr E(-\mu(R_1))\cup\mathscr E(-\mu(R_5))$.
Therefore, $7_6^{1+}\ne7_6^{2+}$.

\begin{rema}
At a very premature stage of the work presented in~\cite{representability,distinguishing}
we expected that whenever rectangular diagrams of a knot~$R$, $R'$
are such that~$\mathscr L_+(R)=\mathscr L_+(R')$ and~$\mathscr L_-(R)=\mathscr L_-(R')$,
and~$\Pi$ is a rectangular diagram of a surface~$\Pi$ with~$\partial\Pi=R$
we must have another rectangular diagram of a surface~$\Pi'$ with~$\partial\Pi'=R'$
having the same dividing code as~$\Pi$ has. To test this expectation, for which we did not have
enough grounds, we
picked the first rectangular diagram~$R$ from~\cite{chong2013}
for which the data of~\cite{chong2013} implied~$\mathscr L_\pm(R)=\mathscr L_\pm(-\mu(R))$
and~$R\ne-\mu(R)$, and this diagram was the~$R_1$ in Figure~\ref{7_6-diagrams-fig} above.
We also constructed a rectangular diagram representing a Seifert surface for~$\widehat R_1$,
which was the~$\Pi_1$ shown in Figure~\ref{loops1-fig}. Then, after searching
all realizations of the dividing code of~$\Pi_1$ we were delighted to see among them
a diagram~$\Pi'$ with~$\partial\Pi'=-\mu(R)$ (which is the top right in Figure~\ref{4realizations-fig}).
This encouraged us to continue this work.

However, as we realized later, the existence of such~$\Pi'$ did not follow from our hypotheses,
and the confirmation of our expectation by this example was accidental
and occurred mainly to the fact that the dividing configurations~$(\delta_1^+,\delta_1^-)$
and~$(\delta_1^+,\sigma(\delta_1^-))$ were isomorphic
(another lucky circumstance was that the diagram~$R$, and hence~$-\mu(R)$,
did not admit any non-trivial exchange move). The point is that~$\Pi'$
is a \emph{proper} realization of~$(\delta_1^+,\sigma(\delta_1^-))$, but not of~$(\delta_1^+,\delta_1^-)$,
whereas our method does not say anything about the use of non-proper realizations
(they may be discarded).
\end{rema}

\noindent\emph{Proof of~(e).}
We follow exactly the same steps as in the proof of the part~(d), so we omit the details except for those
that are different in this case. We now use the Seifert surface
for~$\widehat R_7$ presented by rectangular diagram~$\Pi_2$
shown in Figure~\ref{seifert2-fig} together with new generators~$y_1,y_2,y_3,y_4$ of
\begin{figure}[ht]
\includegraphics[scale=.65]{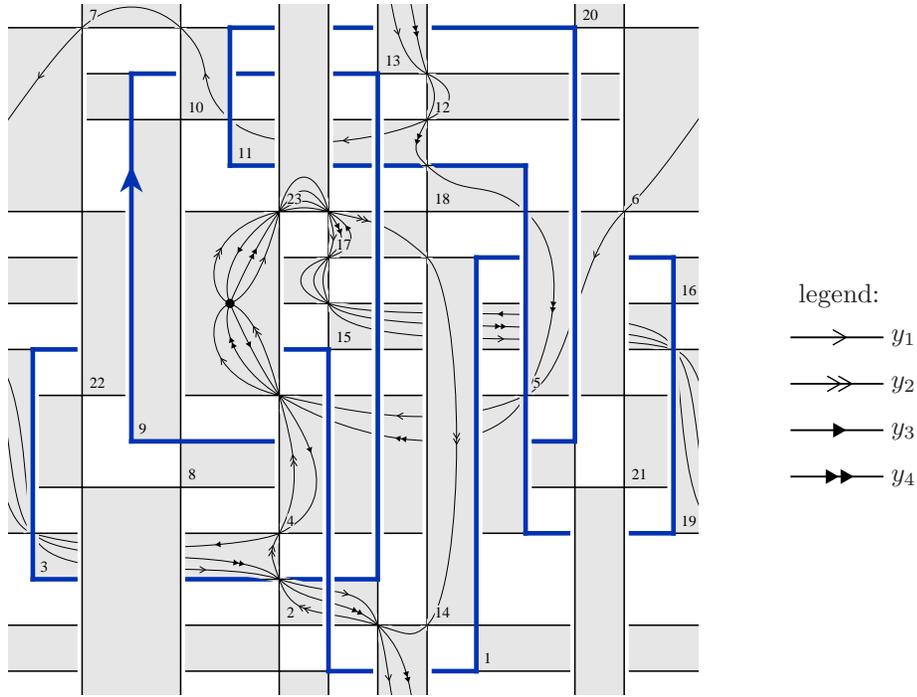}\hskip1cm\raisebox{80pt}{\includegraphics[scale=0.7]{legend.eps}\put(-32,75){legend:}%
\put(3,60){$y_1$}\put(3,42.3){$y_2$}\put(3,24.6){$y_3$}\put(3,6.9){$y_4$}}
\caption{Rectangular diagram~$\Pi_2$ with~$\partial\Pi_2=R_7$}\label{seifert2-fig}
\end{figure}
the fundamental group of~$\widehat\Pi_2$.
A complete set of generators of~$\pi_1(\mathbb S^3\setminus\widehat R_7)$ is shown in Figure~\ref{loops2-fig},
\begin{figure}[ht]
\includegraphics[scale=.8]{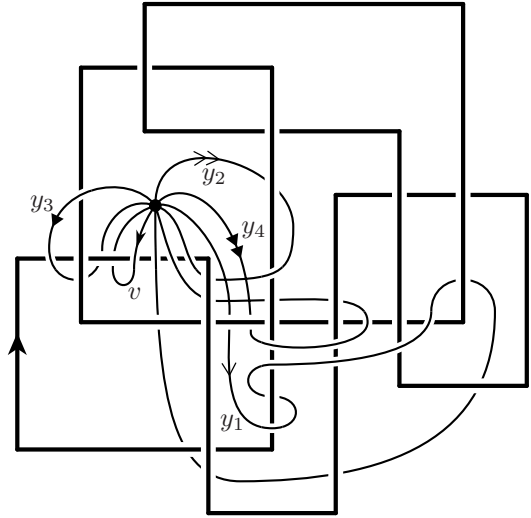}\put(-163,85){$v$}\put(-128,37){$y_1$}\put(-135,130){$y_2$}\put(-200,120){$y_3$}\put(-120,110){$y_4$}
\caption{The generators~$y_1$, $y_2$, $y_3$, $y_4$, and~$v$ of~$\pi_1(\mathbb S^3\setminus\widehat R_7)$}\label{loops2-fig}
\end{figure}
which can be used to verify the following defining relations:
$$v^{-1}y_1v=y_1y_4^{-1}y_1y_3,\quad v^{-1}y_2v=y_3^{-1}y_2y_3y_1y_3y_2,
\quad v^{-1}y_3v=y_2y_3,\quad v^{-1}y_4v=y_3^{-1}y_2^{-2}.$$
One can see from this that, for a smart choice of the involution~$\sigma$, we will have
$$\sigma_*(v)=vy_3^{-1},\quad\sigma_*(y_i)=y_i^{-1},\quad i=1,2,3,4.$$

We denote by~$(\delta_2^+,\delta_2^-)$ a canonic dividing configuration of~$\widehat\Pi_2$.
After cutting the surface along the loops~$y_i$, $i=1,2,3,4$, this configuration
looks as shown in Figure~\ref{nuts2-fig} on the left.
\begin{figure}[ht]
\includegraphics[scale=.5]{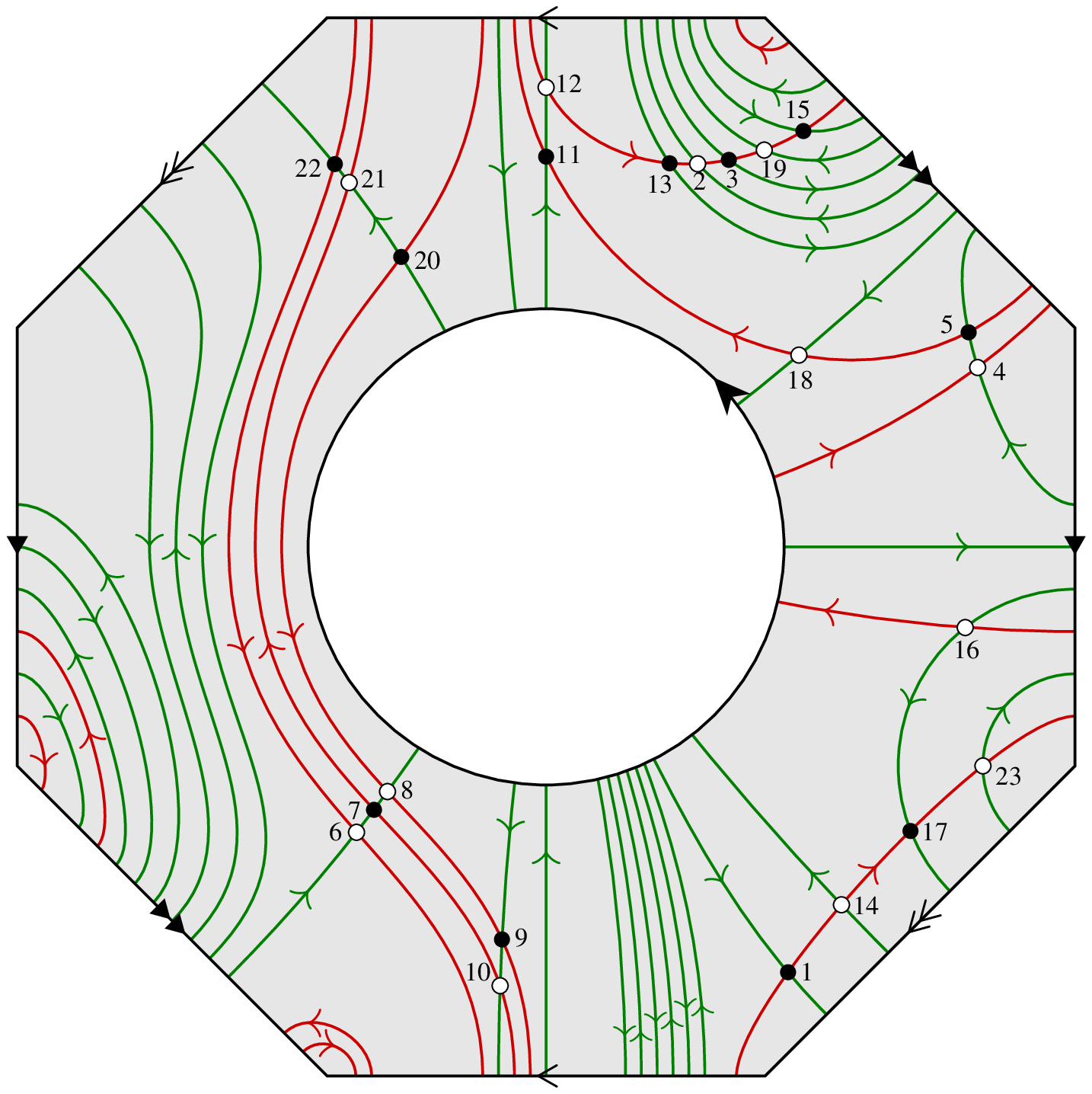}
\includegraphics[scale=.5]{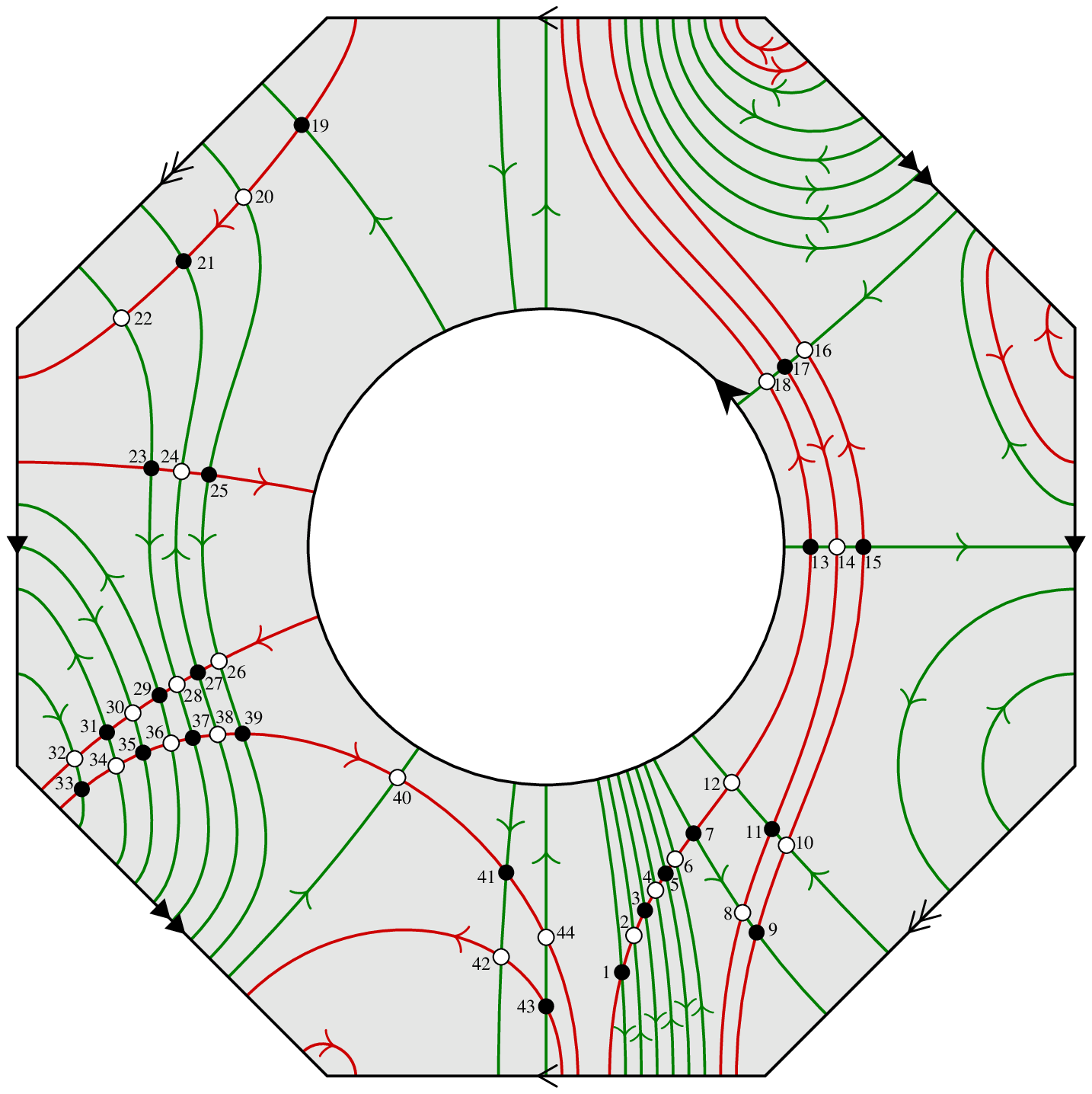}
\caption{Dividing configurations~$(\delta_2^+,\delta_2^-)$ and~$(\delta_2^+,\sigma(\delta_2^-))$}\label{nuts2-fig}
\end{figure}
The right picture in Figure~\ref{nuts2-fig} shows the dividing configuration~$(\delta_2^+,\sigma(\delta_2^-))$
(for a concrete choice of~$\sigma$). The dividing configurations~$(\delta_2^+,\delta_2^-)$
and~$(\delta_2^+,\sigma(\delta_2^-))$
have the following dividing codes, respectively:
\begin{multline}\label{dc-1-eq}
\{(1,2),(3,4,5,6,7,8),(9,10),(11,12),(13,14),(15,16,17,18),(19),(20,21,22,23)\},\\
\{(4,22,6,20,8,9,12,13,2,3,19,15,16),(5,18,11,10,7,21,5),(14,17,23,1,14)\}
\end{multline}
and
\begin{multline}\label{dc-2-eq}
\{
(19,32,33,6),(7,8,9,22,23,28,37,2),(3,36,29,40),(41,42),(43,44),\\
(1,38,27,24,21,10,11,12),(13,14,15,30,35,4),(5,34,31,20,25,26,39,16,17,18)\},\\
\{(19,20,21,22,19),(33,34,35,36,37,38,39,40,41,44,17,14,11,8,33),\\
(26,27,28,29,30,31,32,9,10,15,16,1,2,3,4,5,6,7,12,13,18,43,42,23,24,25)
\}.
\end{multline}

The dividing code~\eqref{dc-1-eq} has exactly three realizations~$(\Pi,\phi)$ with~$\widehat{\partial\Pi}$
isotopic to~$7_6$. For all of them we have~$\partial\Pi\in\mathscr E(R_7)$.

The dividing code~\eqref{dc-2-eq} has exactly $12$ realizations~$(\Pi,\phi)$ with~$\widehat{\partial\Pi}$
isotopic to~$7_6$ (the script~\cite{dyn-script} produces 20 realizations for this dividing code, but in 8 cases
the boundary~$\partial\Pi$ represents the connected sum~$3_1\#4_1$). One of these is shown in Figure~\ref{44rectangles-fig}.
\begin{figure}[ht]
\includegraphics[scale=.65]{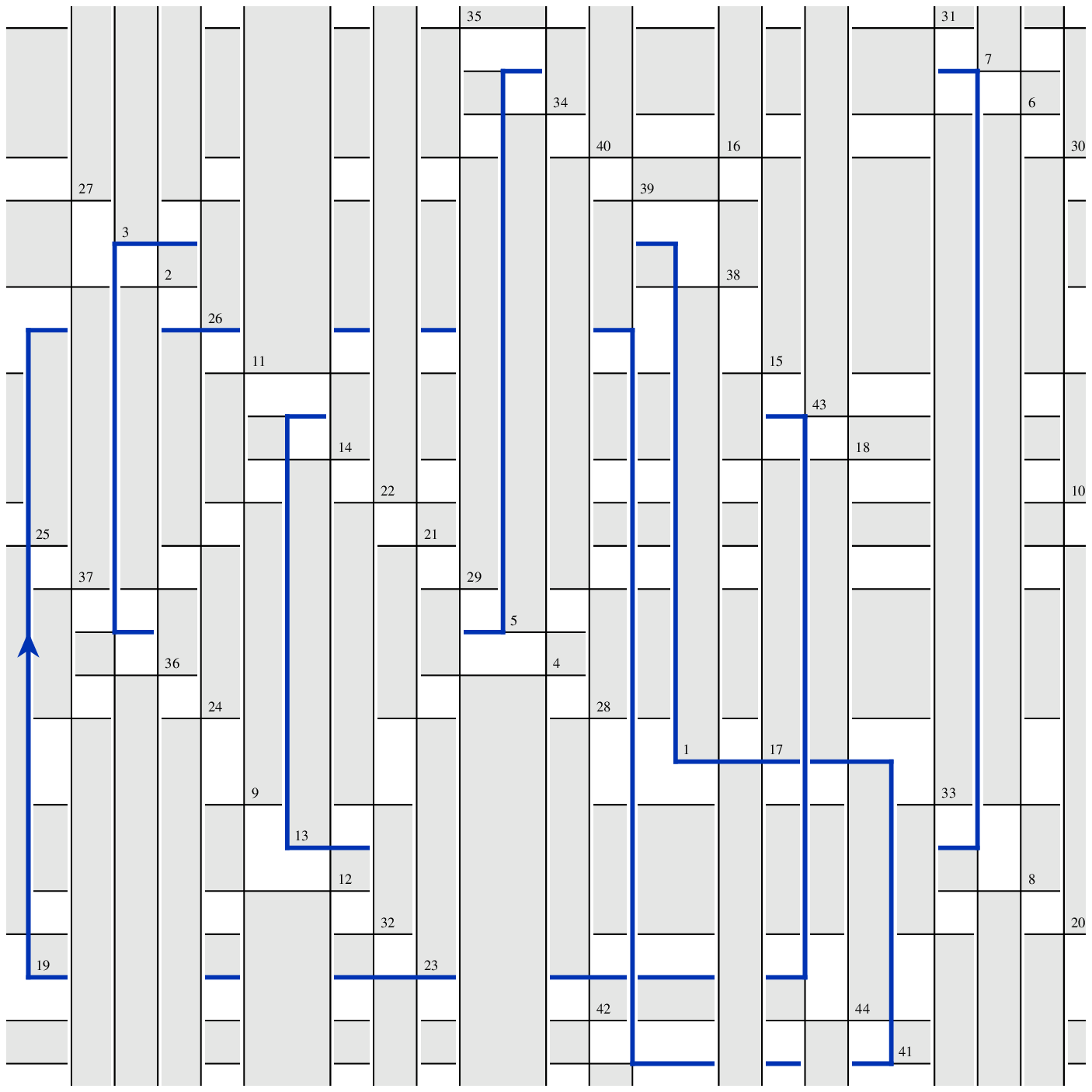}
\caption{A realization of~$(\delta_2^+,\sigma(\delta_2^-))$}\label{44rectangles-fig}
\end{figure}
In all~$12$ cases we again have~$\partial\Pi\in\mathscr E(R_7)$.

Thus, we need not bother to check which of the realizations are proper as it follows
from what was just said and the results of~\cite{distinguishing} that
\begin{equation}\label{implication-eq}
\text{the conditions }\mathscr L_+(R)=\mathscr L_+(R_7)\text{ and }
\mathscr L_-(R)=\mathscr L_-(R_7)\text{ holding simultaneously imply }R\in\mathscr E(R_7).
\end{equation}

It is a direct check that
$$\mathscr L_+(R_3)=7_6^{3+},\quad
\mathscr E\bigl(S_{\overrightarrow{\mathrm I}}(\mu(R_3))\bigr)=
\mathscr E\bigl(S_{\overrightarrow{\mathrm I}}(R_7)\bigr),\quad\text{and}\quad
\mathscr E\bigl(S_{\overleftarrow{\mathrm{II}}}(R_7)\bigr)=
\mathscr E\bigl(S_{\overleftarrow{\mathrm{II}}}(-\mu(R_7))\bigr).$$
This implies
$$\mathscr L_+(R_7)=\mu(7_6^{3+})\quad\text{and}\quad
\mathscr L_-(R_7)=\mathscr L_-(-\mu(R_7)).$$
On the other hand, we have~$\mathscr E(R_7)\ne\mathscr E(-\mu(R_7))$.
Therefore,
$$\mu(7_6^{3+})=\mathscr L_+(R_7)\ne\mathscr L_+(-\mu(R_7))=-7_6^{3+},$$
which implies~(e).
\end{proof}

\begin{prop}
The following is a complete list, without repetitions, of $\xi_-$-Legendrian classes of
topological type~$7_6$ that have maximal possible Thurston--Bennequin number (which is~$-1$):
\begin{equation}\label{list2-eq}
7_6^{1-}=\mu(7_6^{1-})= -7_6^{1-}=-\mu(7_6^{1-}),\ 7_6^{2-}=-\mu(7_6^{2-}),\ -7_6^{2-}=\mu(7_6^{2-}),\
7_6^{3-}=-\mu(7_6^{3-}),\ -7_6^{3-}=\mu(7_6^{3-}).
\end{equation}
\end{prop}

\begin{proof}
It is established in~\cite{chong2013} that the list is complete
and the classes are pairwise distinct except for~$7_6^{2-}$ and~$\mu(7_6^{2-})=-7_6^{2-}$,
which may be coincident. So, we only need to show that~$7_6^{2-}\ne\mu(7_6^{2-})$.

By a direct check we find: $\mathscr L_-(r_\medvert(R_4))=7_6^{2-}$,
$$
\mathscr E\bigl(S_{\overleftarrow{\mathrm I}}(R_7)\bigr)=
\mathscr E\bigl(S_{\overleftarrow{\mathrm I}}(R_8)\bigr),\quad
\mathscr E\bigl(S_{\overleftarrow{\mathrm{II}}}(R_7)\bigr)=
\mathscr E\bigl(S_{\overrightarrow{\mathrm{II}}}(r_{\medvert}(R_4))\bigr),\quad
\mathscr E\bigl(S_{\overrightarrow{\mathrm{II}}}(R_8)\bigr)=
\mathscr E\bigl(S_{\overrightarrow{\mathrm{II}}}(\mu(r_\medvert(R_4)))\bigr),$$
which imply
$$\mathscr L_+(R_7)=\mathscr L_+(R_8),\quad
\mathscr L_-(R_7)=7_6^{2-},\quad\text{and}\quad
\mathscr L_-(R_8)=\mu(7_6^{2-}).$$
Since $\mathscr E(R_7)\ne\mathscr E(R_8)$, it follows from~\eqref{implication-eq}
that~$7_6^{2-}\ne\mu(7_6^{2-})$.
\end{proof}

\end{document}